\documentclass[12pt]{amsart}

\usepackage{etex}
\usepackage[english]{babel}
\usepackage[cp1250]{inputenc}
\usepackage{amssymb,amsthm,amsfonts,amsmath}
\usepackage{mathrsfs}
\usepackage{verbatim}
\usepackage{ulem}
\usepackage{mathtools}

\usepackage{url}
\usepackage{amsopn}
\usepackage{graphicx}
\usepackage[usenames, dvipsnames]{xcolor}
\DeclareGraphicsExtensions{.pdf,.png,.jpg}

\usepackage[shortlabels]{enumitem}
\usepackage{titletoc}
\definecolor{darkblue}{rgb}{0.0, 0.0, 0.55}
\usepackage[colorlinks,linkcolor=BrickRed,citecolor=OliveGreen,urlcolor=darkblue,hypertexnames=true]{hyperref}
\usepackage[a4paper,margin=2.5cm]{geometry}
\usepackage[baseline]{euflag}

\allowdisplaybreaks

\usepackage{tikz}
\usepackage{tikz-3dplot}

\tdplotsetmaincoords{80}{45}
\tdplotsetrotatedcoords{-90}{180}{-90}

\tikzset{surface/.style={fill opacity=.4}}

\newcommand{\coneback}[4][]{
  \draw[draw=blue!70!black, fill=blue!40!white, fill opacity=.6,canvas is xy plane at z=#2, #1] (45-#4:#3) arc (45-#4:225+#4:#3) -- (O) --cycle;
  }
\newcommand{\conefront}[4][]{
  \draw[draw=blue!70!black, fill=blue!40!white, fill opacity=.6,canvas is xy plane at z=#2, #1] (45-#4:#3) arc (45-#4:-135+#4:#3) -- (O) --cycle;
  }

\newcommand{\conebackred}[4][]{
  \draw[draw=red!70!black, fill=red!40!white, fill opacity=.6, canvas is xy plane at z=#2, #1] (45-#4:#3) arc (45-#4:225+#4:#3) -- (O) --cycle;
  }
\newcommand{\conefrontred}[4][]{
  \draw[draw=red!70!black, fill=red!40!white, fill opacity=.6, canvas is xy plane at z=#2, #1] (45-#4:#3) arc (45-#4:-135+#4:#3) -- (O) --cycle;
  }

\DeclareMathOperator{\symm}{Diff}

\DeclareMathOperator{\Pos}{POS}
\DeclareMathOperator{\coPos}{COP}
\DeclareMathOperator{\cop}{COP}
\DeclareMathOperator{\COP}{COP}
\DeclareMathOperator{\DNN}{DNN}
\DeclareMathOperator{\psd}{PSD}

\DeclareMathOperator{\SPN}{SPN}
\DeclareMathOperator{\nn}{NN}
\DeclareMathOperator{\Sq}{SOS}
\DeclareMathOperator{\Lf}{LF}
\DeclareMathOperator{\Sos}{SOS}
\DeclareMathOperator{\Vol}{Vol}

\DeclareMathOperator{\CP}{CP}
\DeclareMathOperator{\cp}{CP}
\DeclareMathOperator{\pr}{pr}

\DeclareMathOperator{\Tr}{tr}
\DeclareMathOperator{\diag}{diag}
\DeclareMathOperator{\vrad}{vrad}
\DeclareMathOperator{\size}{size}

\newcommand{\CR}[1]{\begin{color}{red}#1\end{color}}

\numberwithin{equation}{section}

\newcommand{\x}{{\tt x}}

\newcommand{\Sym}{\mathbb{S}}

\newcommand{\RR}{\mathbb R}

\newcommand{\NN}{\mathbb N}

\newcommand{\sym}{\mathbb S}

\newcommand{\cC}{\mathcal C}

\newcommand{\cH}{\mathcal H}

\newcommand{\cL}{\mathcal L}
\newcommand{\cQ}{\mathcal Q}

\newcommand{\dd}{{\rm d}}

\newcommand{\cM}{\mathcal M}

\newtheorem{theorem}{Theorem}[section]
\newtheorem{corollary}[theorem]{Corollary}
\newtheorem{lemma}[theorem]{Lemma}
\newtheorem{proposition}[theorem]{Proposition}

\theoremstyle{definition}
\newtheorem{definition}[theorem]{Definition}
\newtheorem{remark}[theorem]{Remark}

\title[A random copositive matrix is CP with positive probability]{
A random copositive matrix is completely positive with positive probability}

\author[I.\ Klep]{Igor Klep${}^{1,Q}$}
\address{Igor Klep, Faculty of Mathematics and Physics, Department of Mathematics, University of Ljubljana, Slovenia}
\email{igor.klep@fmf.uni-lj.si}
\thanks{${}^1$Supported by the Slovenian Research Agency 
program P1-0222 and grants J1-50002,
 J1-2453, N1-0217 and J1-3004.}

\author[T.\ \v Strekelj]{Tea \v Strekelj${}^2$}
\address{Tea \v Strekelj, Institute of Mathematics and Physics, Ljubljana, Slovenia
   }
   \email{tea.strekelj@fmf.uni-lj.si}
\thanks{${}^Q$This work was performed within the project COMPUTE, funded within the QuantERA II Programme that has received funding from the EU's H2020 research and innovation programme under the GA No 101017733 {\normalsize\euflag}}

\author[A. Zalar]{Alja\v z Zalar${}^{3,Q}$}
\address{Alja\v z Zalar, 
Faculty of Computer and Information Science, University of Ljubljana  \& 
Faculty of Mathematics and Physics, University of Ljubljana  \&
Institute of Mathematics, Physics and Mechanics, Ljubljana, Slovenia.}
\email{aljaz.zalar@fri.uni-lj.si}
\thanks{${}^3$Supported by the Slovenian Research Agency 
program P1-0228 and grants J1-50002, J1-2453, J1-3004.}


\date{\today}

\begin{document}

\setcounter{tocdepth}{3}
\contentsmargin{2.55em}
\dottedcontents{section}[3.8em]{}{2.3em}{.4pc}
\dottedcontents{subsection}[6.1em]{}{3.2em}{.4pc}
\dottedcontents{subsubsection}[8.4em]{}{4.1em}{.4pc}

\makeatletter
\newcommand{\mycontentsbox}{%
{
\addtolength{\parskip}{.3pt}
\tableofcontents}}
\def\enddoc@text{\ifx\@empty\@translators \else\@settranslators\fi
\ifx\@empty\addresses \else\@setaddresses\fi
\newpage\mycontentsbox\newpage\printindex}
\makeatother

\begin{abstract}
	An $n\times n$ symmetric matrix $A$ is copositive if the quadratic form $x^TAx$	
	is nonnegative on the nonnegative orthant $\RR^{n}_{\geq 0}$. 
  	The cone of copositive matrices strictly contains the cone of 
        completely positive matrices, i.e., all matrices
        of the form $BB^T$ for some $n\times r$ matrix $B$ with nonnegative entries. 
        The main result, 
        proved using Blekherman's real algebraic geometry inspired techniques 
        and tools of convex geometry, 
        shows that asymptotically, as $n$ goes to infinity, 
        the ratio of volume radii of the two cones is strictly positive.
        Consequently, the same holds true for the ratio
        of volume radii of any two cones sandwiched between them,
        e.g., the cones of positive semidefinite matrices, matrices with nonnegative entries, their intersection and their Minkowski sum. \looseness=-1
\end{abstract}

\maketitle

\section{Introduction}\label{introd}
Copositive and completely positive matrices arise in many areas, including block designs in combinatorial analysis, complementarity problems in computational mechanics, exchangeable probability distributions \cite{BDSM15, BSM21}, more recently even in data mining and clustering \cite{DHS05} as well as in dynamical systems and control theory \cite{MS07}.
Further, many combinatorial and nonconvex quadratic optimization problems can be formulated as linear problems over the cones of either copositive or completely positive matrices. Thus 
also in  mathematical optimization copositive and completely positive matrices have received considerable attention in recent years. 
This area is called copositive (resp., completely positive) programming \cite{Bom12, Dur10}, and  can be seen as a generalization of semidefinite programming, where the cone of positive semidefinite matrices is replaced by the cone of copositive (resp., completely positive) matrices. In this paper we estimate the asymptotic volumes of these two cones of matrices.

\subsection{Notation}
    For $n\in \NN$ let $M_n(\RR)$ be the vector space of $n\times n$ real matrices 
    and let $\Sym_n=\left\{ A\in M_n(\RR)\colon A^T=A\right\}$ be its subspace of real symmetric matrices,
    where $^T$ stands for the usual transposition of matrices.
    Let $\RR[\x]$ be the vector space of real polynomials in the variables $\x=(x_1,\ldots,x_n)$, and
     $\RR[\x]_k$ its subspace of \textbf{forms of degree} $k$, i.e., 
    homogeneous polynomials from $\RR[\x]$ of degree $k$.
    To a matrix 
        $A=[a_{ij}]_{i,j=1}^n\in \Sym_n$ 
    we associate the quadratic form 
    \begin{equation}
	\label{correspondence-quadratic}
		p_A(\x):=\x^T A\x=\sum_{i,j=1}^n a_{ij}x_ix_j\in \RR[\x]_2.
	\end{equation}

\subsection{Basic definitions}	The main goal of this paper is to estimate asymptotic volumes (as $n$ goes to infinity) of the cones of the following classes of matrices.

\begin{definition}
\label{def-cones}
    A matrix $A\in \Sym_n$ is:
    \begin{enumerate}
    \item 
    \textbf{copositive} 
    if $p_A$ is nonnegative  
	on the nonnegative orthant 
		$$\RR^{n}_{\geq 0}:=\{(x_1,\ldots,x_n)\colon x_i\geq 0,\ i=1,\ldots,n\},$$
	i.e., $p_A(\x)\geq 0$ for every $\x\in\RR^{n}_{\geq 0}$.
    Equivalently, $A$ is copositive iff the quartic form 
     \begin{equation}\label{eq:q_A}
     q_A(\x):=p_A(x_1^2,\ldots,x_n^2)\in \RR[\x]_4
     \end{equation}
    is nonnegative on $\RR^n$.        
    We write $\coPos_n$ for the cone of all $n\times n$ copositive matrices.
    \item
    \textbf{positive semidefinite (PSD)}	 
        if all of its eigenvalues are nonnegative.
    Equivalently, $A$ is PSD 
        iff $p_A(\x)\geq 0$ for all $\x\in \RR^n$
        iff $A=BB^T$ for some matrix $B\in M_n(\RR)$.
	We write $A\succeq 0$ to denote that $A$ is PSD and $\psd_n$ stands for the cone of all $n\times n$ PSD matrices.
    \item   
        \textbf{nonnegative (NN)}
        if all of its entries are nonnegative, i.e., 
        $A=[a_{ij}]_{i,j=1}^n$ with 
        $a_{ij}\geq 0$ for $i,j=1,\ldots,n$.
    We write $\nn_n$ for the cone of all $n\times n$ NN matrices.
    \item 
    \textbf{SPN}
    (sum of a positive semidefinite matrix and a nonnegative one)
        if it is of the form $A=P+N$, 
        where 
            $P\in \psd_n$
        and 
            $N\in \nn_n$.
    We write $\SPN_n:=\psd_n+\nn_n$ for the cone of all $n\times n$ SPN matrices.
    \item   
    \textbf{doubly nonnegative (DNN)}
        if it is PSD and NN.
    We write $\DNN_n:=\psd_n\cap \nn_n$ for the cone of all $n\times n$ DNN matrices.
    \item   
    \textbf{completely positive (CP)}{\footnote{Despite the similar name, the CP matrices considered here are not related to the CP maps ubiquitous in operator algebra \cite{Pau02}.}}
        if $A=BB^T$ for some $r\in\NN$ and $n\times r$ entrywise nonnegative matrix $B$.
    We write $\cp_n$ for the cone of all $n\times n$ CP matrices.
    \end{enumerate}
\end{definition}
    Clearly,
    \begin{equation}
        \label{101222-1515}
            \coPos_n
            \supseteq 
                \SPN_n
            \supseteq 
                \psd_n\cup\nn_n
            \supseteq
                \DNN_n
            \supseteq
                \cp_n.
    \end{equation}
    

\subsection{Main result}
\label{subsec-main-results}
Let $V$ be a finite-dimensional Hilbert space equipped with the pushforward measure of the Lebesgue measure on $\RR^{\dim V}$. The \textbf{volume radius} $\vrad(C)$ of a
compact measurable set $C$ 
is defined by\looseness=-1
	$$\vrad(C)=\left(\frac{\Vol(C)}{\Vol(B)}\right)^{1/\dim V},$$
where $B$ is the unit ball in $V$.
(See Remark \ref{volumes-vs-vrads} below for a discussion of the meaning of $\vrad$.)

To obtain compact sets from our convex cones in \eqref{101222-1515} we shall intersect them with suitable hyperplanes.
In the usual Frobenius inner product on $\sym_n$ (i.e., $\langle A,B\rangle=\Tr(AB)$),
one of the most natural choices of a hyperplane is either a constant trace hyperplane $\cH_{\Tr}$, e.g., $\cH_{\Tr}:=\{A\in \sym_n\colon \Tr(A)=1\}$, or a hyperplane $\cH_{\sum}$ of a constant sum of all the entries, e.g., 
$\cH_{\sum}:=\{[a_{ij}]_{i,j=1}^n\in \sym_n\colon \sum_{i,j=1}^n a_{ij}=1\}$. 
However, the intersections 
$\COP_n\cap \cH_{\Tr}$ and $\COP_n\cap \cH_{\sum}$ are not bounded.
So it is not even immediately clear what hyperplane to choose such that all the intersections with the cones in question are bounded, let alone the question of 
`fairness' for size difference estimates. (Discussed in detail in \S\ref{choice-hyperplane} below.)
To obtain a suitable hyperplane such that the intersections with the cones in \eqref{101222-1515} are compact and are fair size difference representatives,
we use the identification \eqref{correspondence} below of $\Sym_n$ 
with the vector subspace of all \textbf{even quartic forms} $\cQ$
in $\RR[\x]_4$, i.e., 
\begin{equation}
    \label{101222-1925}
        \cQ:=
            \big\{
                f\in \RR[\x]_4\colon 
                f(\x)=\sum_{i,j=1}^n a_{ij}x_i^2x_j^2\quad
                \text{where each }a_{ij}\in \RR \text{ and }
                \forall i,j:
                a_{ij}=a_{ji} 
            \big\}.
\end{equation}
We thus have a bijective 
correspondence
    \begin{equation}
	\label{correspondence}
    \Phi:\Sym_n \to \cQ,\qquad 
    A=[a_{ij}]_{i,j=1}^n 
    \mapsto 
    q_A(\x)=p_A(x_1^2,\ldots,x_n^2)\in \RR[\x]_4,
   \end{equation}
   where $p_A$ is as in 
   \eqref{correspondence-quadratic}.
One of the natural choices of an inner product on $\RR[\x]_4$ is the $L^2$ 
inner product, given by 
\begin{equation}
\label{L2-inner-product-intro}
\langle f,g\rangle=\int_{S^{n-1}}fg\; \dd\sigma,
\end{equation}
where $S^{n-1}$ is the unit sphere in $\RR^n$ 
and 
$\sigma$ is the rotation invariant probability measure on $S^{n-1}$.
A suitable choice of the hyperplane in $\RR[\x]_4$ (see \S\ref{choice-hyperplane}) 
is the affine hyperplane $\cL$ of forms from $\RR[\x]_{4}$ of average 1 on $S^{n-1}$, i.e.,
\begin{equation}
\label{hyperplane-forms-intro}
    \cL=
    \left\{ f\in \RR[\x]_{4}\colon \int_{S^{n-1}} f\; \dd\sigma=1 \right\}.
\end{equation}
Let $\langle \cdot,\cdot \rangle_{\Phi}$
be the inner product on 
$\sym_n$
making
$\Phi$  a Hilbert space isomorphism.
(See Remark \ref{correspondence-discussion} below for a concrete presentation of this inner product.)
It turns out (see Lemma \ref{090123-1550} below),
that $\cL_{\Sym_n}:=\Phi^{-1}(\cL)$ is defined by
\begin{align}
\label{hyperplanes-matricial}
\begin{split}
    \cL_{\Sym_n}
    &=
    \left\{ 
        [a_{ij}]_{i,j=1}^n\in \Sym_n
        \colon  
        3\Big(\sum_{i=1}^n a_{ii}\Big)+\sum_{i\neq j} a_{ij}=n(n+2)
    \right\}\\[0.2em]
    &=
    \left\{ 
        [a_{ij}]_{i,j=1}^n\in \Sym_n
        \colon 
        \langle
        [a_{ij}]_{i,j=1}^n,
        \mathbf{1}_{n,n}
        \rangle_{\Phi}
        =n(n+2)
    \right\},
\end{split}
\end{align}
where 
$\mathbf{1}_{n,n}$
is
the matrix of all ones.
Let $\cM_{\Sym_n}$ 
be the hyperplane 
in $\Sym_n$ defined by
\begin{align}
\label{hyperplanes-matricial-v2}
\begin{split}
    \cM_{\Sym_n}
    &=
    \left\{ 
        [a_{ij}]_{i,j=1}^n\in \Sym_n
        \colon  
        3\Big(\sum_{i=1}^n a_{ii}\Big)+\sum_{i\neq j} a_{ij}=0
    \right\},\\[0.2em]
    &=
    \left\{ 
        [a_{ij}]_{i,j=1}^n\in \Sym_n
        \colon 
        \langle
        [a_{ij}]_{i,j=1}^n,
        \mathbf{1}_{n,n}
        \rangle_{\Phi}
        =0
    \right\}.
\end{split}
\end{align}
With respect to the inner product $\langle\cdot,\cdot\rangle_{\Phi}$, $\cM_{\Sym_n}$ is a subspace of $\Sym_n$ of dimension $\dim\cM_{\Sym_n}=\frac{n(n+1)}{2}-1$ 
and so it is isomorphic to $\RR^{\dim{\cM_{\Sym_n}}}$ as a Hilbert space.
Let $S_{\cM_{\Sym_n}}$, $B_{\cM_{\Sym_n}}$ be the unit sphere and the unit ball in $\cM_{\Sym_n}$, respectively. 
Let $\psi:\RR^{\dim\cM_{\Sym_n}}\to \cM_{\Sym_n}$ be a unitary isomorphism and $\psi_{\ast}\mu$ the pushforward of the
Lebesgue measure $\mu$ on $\RR^{\dim\cM_{\Sym_n}}$ to $\cM_{\Sym_n}$, i.e., $\psi_\ast\mu(E):=\mu(\psi^{-1}(E))$ for every
Borel measurable set $E\subseteq \cM_{\Sym_n}$.

\begin{lemma} 
    \label{unique-pushforward}
	The measure of a Borel set $E\subseteq \cM_{\Sym_n}$ does not depend on the choice of the unitary isomorphism $\psi$, i.e.,
	if $\psi_1:\RR^{\dim\cM_{\Sym_n}}\to \cM_{\Sym_n}$ and $\psi_2:\RR^{\dim\cM_{\Sym_n}}\to \cM_{\Sym_n}$ are unitary isomorphisms, then
	$(\psi_1)_\ast\mu(E)=(\psi_2)_\ast\mu(E)$.
\end{lemma}

\begin{proof}
The proof of Lemma \ref{unique-pushforward} is routine and the same as the proof of 
\cite[Lemma 1.4]{KMSZ19}.
\end{proof}

We will compare the sizes of the convex cones $K$ from Definition \ref{def-cones} 
by comparing the volumes of their intersections 
with $\cL_{\Sym_n}$,
\begin{equation}
    \label{compact-section}
    K'_{\Sym_n}=K_{\Sym_n}\cap \cL_{\Sym_n},
\end{equation} 
when translated to $\cM_{\Sym_n}$ by subtracting $\mathbf{1}_{n,n}$, i.e.,
\begin{equation}
    \label{compact-section-translated}
\widetilde K_{\Sym_n}
:=K'_{\Sym_n}-\mathbf{1}_{n,n}
=\left\{
[a_{ij}]_{i,j=1}^n\colon [a_{ij}+1]_{i,j=1}^n\in K'_{\Sym_n}
\right\}. 
\end{equation}
Note that for all $K_{\Sym_n}$ from Definition \ref{def-cones},  
the section $K'_{\Sym_n}$ 
is a convex, compact, full-dimensional set
in the finite-dimensional {affine} hyperplane $\cL_{\Sym_n}$.

\begin{remark}
\label{boundedness-sections}
An argument for boundedness of $\COP_n'$ (and hence also all smaller cones) is as follows. 
Let $A:=[a_{ij}]_{i,j=1}^n\in \COP_n'$, 
$\mathbf{1}_n\in \RR^n$ be the vector of ones
and
$e_i\in \RR^n$ the standard coordinate vector with the only nonzero coordinate at position $i$, which is equal to 1.
Since $A$ is copositive, we have in particular that
$\mathbf{1}_n^TA\mathbf{1}_n=\sum_{i,j=1}^n a_{ij}\geq 0$
and
$e_i^TAe_i=a_{ii}\geq 0$ 
for each $i$.
Using this and the definition of $\cL_{\sym_n}$,
it follows that 
$0\leq a_{ii}\leq \frac{n(n+2)}{2}$ for each $i$. Further, copositivity of $A$ implies that
$(e_i+e_j)^TA(e_i + e_j)=a_{ii}+a_{jj}+ a_{ij}+ a_{ji}\geq 0$ 
for each $i\neq j$. Hence, $a_{ij}=a_{ji}\geq -\frac{n(n+2)}{2}$ for each $i\neq j$. 
But then also 
\begin{align*}
a_{ij}=
\frac{1}{2}\Big(
n(n+2)-3\sum_{i=1}^n a_{ii}
-\sum_{\substack{k\neq \ell,\\
\{k,\ell\}\neq \{i,j\}
}}a_{k\ell}\Big)
&\leq \frac{n(n+2)}{2}+(n^2-n-2)\frac{n(n+2)}{4}\\
&=\frac{1}{4}n(n+2)(n^2-n+2),
\end{align*}
where in the first inequality we used that $a_{ii}\geq 0$
and $a_{k\ell}\geq -\frac{n(n+2)}{2}$ for each $i$ and each $k\neq \ell$.
Therefore $A$ has all entries bounded in the absolute value and thus $\COP'_n$ is bounded.
\end{remark}

The main result of the paper is as follows.

\begin{theorem}
    \label{intro-vrad-of-our-sets} 
    We have that
    \begin{align*}
            \frac{1}{2^4\sqrt{2}}\cdot\frac{1}{n}
            &\leq 
            \vrad(\widetilde{\CP_n})
            \leq
            \vrad(\widetilde{\DNN_n})
            \leq
            \left\{
            \begin{array}{c}
            \vrad(\widetilde{\psd_n})\\[0.2em]
            \vrad(\widetilde{\nn_n})
            \end{array}\right.\\[0.3em]
            &\leq
            \vrad(\widetilde{\SPN_n})
            \leq 
            \vrad(\widetilde{\COP_n})
            \leq 
            2^3\cdot 3^2\cdot \sqrt{2}\cdot\frac{1}{n}.
    \end{align*}
    In particular,
    $$
    \frac{1}{2^{8}\cdot 3^2}
    \leq
    \frac
    {\vrad(\widetilde{\CP_n})}{\vrad(\widetilde{\COP_n})}
    \leq 
    1.
    $$
\end{theorem}

\begin{remark}
\begin{enumerate}
    \item 
        There are two possible  interpretations of the statements of Theorem \ref{intro-vrad-of-our-sets}. One is that the section $\widetilde{\CP_n}$
        is relatively large, while the other is that $\widetilde{\COP_n}$
        is relatively small. 
        We believe that the first interpretation is more likely, as we now explain.
        Let $B_{n,\max}$ be the unit ball in the max
        norm 
        $\|[a_{ij}]_{i,j=1}^n\|_{\max}=
        \max_{i,j} |a_{ij}|$
        on $\sym_n$. 
        (This is not a proper matrix norm, since it is not submultiplicative.)
        Pushforwarding the Lebesgue measure on $\RR^{\dim \sym_n}$
        to $\sym_n$, equipped with the usual Frobenius inner product,  it is clear that
        $\Vol(\nn_n\cap B_{n,\max})=
        (\frac{1}{2})^{\dim \sym_n}\cdot \Vol(B_{n,\max})$
        and hence 
        $\Vol(\COP_n\cap B_{n,\max})=
        c^{\dim \sym_n}\cdot \Vol(B_{n,\max})$
        for some $c\in [\frac{1}{2},1]$.
        It follows that $\COP_n\cap B_{n,\max}$ occupies a large portion of $B_{n,\max}$.
        Theorem \ref{intro-vrad-of-our-sets}
        suggests that
        also the portion of $B_{n,\max}$,
        which is occupied by $\CP_n\cap B_{n,\max}$, does not get arbitrarily small as $n\to\infty$, which is more surprising.
        However, this does not follow directly from Theorem \ref{intro-vrad-of-our-sets}, but remains as an open problem for future work. Comparing sizes of cones depends on the inner product and the compact sections chosen, and it is conceivable that in $B_{n,\max}$ the difference between the cones becomes visible (see also \S\ref{choice-hyperplane} below).
    \item Establishing precise constants to bound 
        the ratios between the volume radii of cones in Theorem \ref{intro-vrad-of-our-sets} would require more precise computations of volumes of the sections and remains an open question. The techniques we use also lean on some general inequalities from harmonic and convex analysis (see \S\ref{the-differential-metric}--\S\ref{RShepard}) and are suitable to determine the asymptotic behavior of sizes only up to a constant.  
\end{enumerate}
\end{remark}

\subsection{Applications and known results about the cones from Definition \ref{def-cones}}
The cones from Definition \ref{def-cones} are of great importance in optimization, 
since many combinatorial problems can be formulated as
conic linear programs over the largest cone $\COP_n$ or the smallest cone $\CP_n$ among them \cite{KP02,Bur09,RRW10,DR21}.
However, deciding whether a given matrix belongs to $\COP_n$ is co-NP-complete \cite{MK87} and NP-hard for $\CP_n$ \cite{DG14}.
To get a tractable approximation of the cone $\COP_n$ based on semidefinite programming, Parrilo \cite{Par00} proposed an increasing hierarchy of inner approximating cones
$K_n^{(r)}:=\{A\in \Sym_n\colon (\sum_{i=1}^n x_i^2)^r \cdot p_A(\x) \text{ is a sum of squares of forms}\}$.
Clearly, 
\begin{equation}
\label{Parrilo-approx}
	\bigcup_{r\in \NN_0}  K_n^{(r)}\subseteq \COP_n,
\end{equation}
and 
by a result of P\'olya \cite{Pol28}, 
	$\mathrm{int}(\COP_n)\subseteq \bigcup_{r\in \NN_0} K_n^{(r)}$.
By \cite[p.\ 63--64]{Par00}, $K_n^{(0)}=\SPN_n$. 
Since $\SPN_n=\COP_n$ for $n\leq 4$ \cite{MM62}, it follows that $K_n^{(0)}=\COP_n$
and the inclusion in \eqref{Parrilo-approx} is an equality for $n\leq 4$ (see also \cite{Dia62}).
For $n\geq 5$, $K_n^{(0)}$ is strictly contained in $\COP_n$; the so-called Horn matrix \cite{HN63}

\begin{equation}
\label{Horn-matrix}
H=
\begin{pmatrix*}[r]
	1 & -1 & 1 & 1 & -1 \\
	-1 & 1 & -1 & 1 & 1 \\
	1 & -1 & 1 & -1 & 1 \\
	1 & 1 & -1 & 1 & -1 \\
	-1 & 1 & 1 & -1 & 1 
\end{pmatrix*}
\end{equation}
is a standard example of a copositive matrix that is not $\SPN$.

Further, $H\in K_5^{(1)}$ \cite{Par00}, but $\COP_5\neq K_5^{(r)}$ for any $r\in \NN$ \cite{DDGH13}.
It has been very recently shown \cite{LV22b,SV+} that for $n=5$ the inclusion in \eqref{Parrilo-approx} is still the equality,
while for $n\geq 6$, the inclusion is strict \cite{LV22a}.
For a nice exposition on the classes of matrices defined above we refer the reader to \cite{BSM21}.
Some open problems regarding $\COP_n$, $\CP_n$ are presented in \cite{BDSM15}.

\begin{remark}
Theorem \ref{intro-vrad-of-our-sets} is thus of interest from a computational complexity viewpoint.
Unlike the Blekherman results in \cite{Ble06} which
imply that in general one cannot replace testing for positivity of a polynomial with testing whether a polynomial is a sum of squares,
our results suggest that in certain problems it could be feasible to replace the larger cone $\COP_n$ with each of the smaller cones
in \eqref{101222-1515} (where the smallest cone $\CP_n$ might again give a potentially equally hard problem to solve).
Further, Theorem \ref{intro-vrad-of-our-sets} can be seen as the first step towards quantifying tightness of the Parrilo $K_n^{(r)}$ 
hierarchy of approximations to $\COP_n$.
\end{remark}

\subsection{Idea of the proof of Theorem
\ref{intro-vrad-of-our-sets}
and the second main result}
To prove Theorem \ref{intro-vrad-of-our-sets} 
we lean on powerful techniques, 
developed by Blekherman \cite{Ble04,Ble06} and Barvinok-Blekherman \cite{BB05}
for comparing the cones of positive polynomials and sums of squares, fundamental objects of real algebraic geometry
\cite{BCR98,Mar08,Lau09,Sce09}.
In addition to these we will rely on some classical results of convex analysis \cite{RS57,BM87,MP90}. Below we briefly explain the idea of the proof.

Let $\cQ$ be as in 
\eqref{101222-1925}.
Under the correspondence
$\Phi$ as in
\eqref{correspondence}, 
$\coPos_n$ 
corresponds to 
nonnegative forms  
\begin{equation}
\label{nonnegative-quartics}
\Pos_\cQ:=
            \{
                f\in \cQ\colon f(\x)\geq 0\quad \text{for all } \x\in \RR^n
            \}
\end{equation}
from $\cQ$, 
$\psd_n$
corresponds to 
sums of squares of quadratic forms in the variables $x_1^2,\ldots,x_n^2$,
$\nn_n$
corresponds to 
forms in $\cQ$ with nonnegative coefficients,
$\SPN_n$
corresponds to 
sums of squares of all quadratic forms which belong to $\cQ$, i.e., 
    \begin{equation}
    \label{101222-2006}
    \Sq_\cQ:=\cQ
            \;\bigcap\;
            \Big\{
                f\in \RR[\x]_4\colon 
                f=\sum_{i} f_i^2\quad 
                \text{for some } f_i\in\RR[\x]_2
            \Big\}
    \end{equation}
and
$\cp_n$
corresponds to
corresponds to 
sums of squares of quadratic forms in variables 
		$x_1^2,\ldots,x_n^2$ 
	with nonnegative coefficients.
Thus comparing the cones from Definition
\ref{def-cones}
is equivalent to comparing
the corresponding cones of forms from $\cQ$.\looseness=-1

To estimate the gap between the cones of forms from $\cQ$ we compare the volumes of compact sections obtained by intersecting each with 
the affine hyperplane $\cL$ of \eqref{hyperplane-forms-intro}.
Note that by construction
this affine hyperplane corresponds to $\cL_{\Sym_n}$ from \eqref{hyperplanes-matricial}
under $\Phi$ and hence the sections correspond to sections $K'_{\Sym_n}$ from \eqref{compact-section}. 
For technical reasons we translate every section $\Phi(K'_{\Sym_n})$ from $\cL$ to become a subset 
$\widetilde{\Phi(K'_{\Sym_n})}$ in the hyperplane $\cM$ of forms from $\cQ$ with average 0 on $S^{n-1}$. 
The vector of translation is the negative of the polynomial $r(\x)=(x_1^2+\ldots+x_n^2)^{2}$,
i.e., 
$\widetilde{\Phi(K'_{\Sym_n})}=
\Phi(K'_{\Sym_n})-r$. (Note that under $\Phi$ the polynomial $r$ corresponds to the matrix $\mathbf{1}_{n,n}$.) 
With respect to the $L^2$ inner product \eqref{L2-inner-product-intro} this hyperplane is isomorphic to 
$\RR^{\dim{\cM_{\Sym_n}}}$ as a Hilbert space and can be equipped with the pushforward of the Lebesgue measure on $\RR^{\dim\cM_{\Sym_n}}$ under some unitary isomorphism.

    It turns out that it is crucial to introduce another subset of $\cQ$ to estimate the size of $\widetilde{\Phi(\cop_n')}$. 
    Namely, the set of all projections 
    to $\cQ$ of sums of 4-th powers of linear forms:
    \begin{equation}
    \label{linear-forms-projections}
    \Lf_\cQ:=
            \Big\{
               \pr_\cQ(f)\in \RR[\x]_4\colon 
                f=\sum_{i} f_i^4\quad 
                \text{for some } f_i\in\RR[\x]_1
            \Big\},
    \end{equation}
    where $\pr_\cQ:\RR[\x]_4\to \cQ$ is defined by:
    \begin{equation}
	\label{proj-Q}
    \pr_\cQ\big(
    \sum_{1\leq i\leq j\leq k\leq \ell\leq n}
    a_{ijk\ell} x_{i}x_jx_kx_\ell
    \big)
    =
    \sum_{1\leq i\leq j\leq n}
    a_{iijj}x_i^2x_j^2.
    \end{equation}
    The translate $\widetilde{\Lf_\cQ}:=\Lf_\cQ'-r$
    of the section $\Lf_\cQ':=\Lf_\cQ\cap \cL$
    of $\Lf_\cQ$ is dual to the section of $\Phi(\cop_n)$ in the so-called 
    differential metric \cite[Section 5]{Ble06}, originally called an apolar inner product (see \cite[p.\ 11]{Rez82} for a historical account).
    To get the needed lower bound on the size of $\widetilde{\Lf_\cQ}$ we compare it to size of the section $\widetilde{\Phi(\nn_n')}$ for which
    the lower bound is obtained by a version of the reverse Blaschke-Santal\'o inequality in the differential metric
	together with a self-duality of $\Phi(\nn_n)$.
    (For the original version of the inequality see \cite{BM87}, while for the version with explicit bounds on the value of the absolute constant appearing we refer to \cite{Kup08}.)
    Finally, to establish the upper bound on size of the section $\widetilde{\Phi(\cop_n)}$
    we use the Blaschke-Santal\'o inequality \cite[p.\ 90]{MP90} after proving that the origin is the Santal\'o point (see \S\ref{BSantalo}) of $\widetilde{\Lf_\cQ}$ (see Lemma \ref{Santalo-of-Lf}).

    In particular, we obtain the second main result of this paper.

\begin{theorem}
    \label{vrad-quartics} 
    We have that
    $$
    \frac{1}{2^{8}\cdot 3^2}
    \leq
    \frac
    {\vrad(\widetilde{\Sos_\cQ})}{\vrad(\widetilde{\Pos_\cQ})}
    \leq 
    1,
    $$
    where
    \begin{align*}
        \widetilde{\Pos_\cQ}&:=
    \left\{ f\in \cQ\colon  
        \int_{S^{n-1}}f\;\dd\sigma=0
        \quad\text{and}\quad
        f+(x_1^2+\ldots+x_n^2)^{2}\in \Pos_{\cQ}\right\},\\
    \widetilde{\Sos_\cQ}&:=
    \left\{ f\in \cQ\colon  
        \int_{S^{n-1}}f\;\dd\sigma=0
        \quad\text{and}\quad
        f+(x_1^2+\ldots+x_n^2)^{2}\in \Sos_{\cQ}\right\}.
    \end{align*}
\end{theorem}

\begin{remark}[Comment on the statement of Theorem \ref{vrad-quartics}]
In \cite{Ble06}
Blekherman established
estimates on the volume radii of compact sections of the cones of nonnegative forms and sums of squares forms. 
For a fixed degree bigger than 2, as the number of variables goes to infinity, the ratio between the volume radii goes to 0. 
In particular, this holds for quartic forms, while
Theorem \ref{vrad-quartics} states that this is not the case for even forms. This is slightly surprising, since even quartic forms are a relatively large subspace in all quartic forms and 
one would have expected that the ratio between the volume radii would follow the same asymptotics.
\end{remark}

\begin{remark}
[Difference with the original work of Blekherman \cite{Ble06}]
\label{subsec-comparison-Blek}
Next we explain the adaptations of Blekherman's original approach to compare the volumes of nonnegative forms and sums of squares forms needed, to give sufficiently tight estimates on the volume radii of the cones 
from Definition \ref{def-cones} we are interested in. 
To study these cones one could a priori use either of the one-to-one correspondences 
\eqref{correspondence-quadratic} or \eqref{correspondence} and then compare sizes of the corresponding cones 
in either quadratic or quartic forms, respectively.
However, in quadratic forms establishing volume estimates does not seem to be straightforward with specializing Blekherman's approach;
it is not clear how to handle copositivity to derive asymptotically tight estimates.
Working with the cones within even quartics using \eqref{correspondence} is more convenient. 
In the latter case one can specialize Blekherman's techniques, but the obtained estimates are not tight enough:
the upper bound on the volume radius of even SOS quartics turns out 
to be larger than the lower bound on nonnegative ones. 
Therefore one has to adapt the methods to obtain tighter estimates. The crucial observation to achieve this is to notice
that dilations of the difference bodies of the compact sections of $\Lf_{\cQ}$, resp.\ $\Phi(\cp_n)$, by absolute constants independent of the dimension, contain the compact section
of $\Phi(\nn_n)$ (see Lemma \ref{dualities-diff-metric} below). This fact together with the reverse Blaschke-Santal\'o inequality (see Corollary \ref{BSr-in-d-metric} below) suffices to obtain
 conclusive estimates.
\end{remark}

\begin{remark}
[Other results related to Theorem \ref{vrad-quartics}]
In \cite{BR21}, the authors compared the cone of nonnegative symmetric quartics and 
the cone of symmetric quartics that are sums of squares of quadratics. Similarly as in the setting of the present paper
and in sharp contrast to \cite{Ble06}, asymptotically the ratio on the volume radii of compact sections of these cones is strictly positive as the number of variables goes to infinity. However, the subspaces of quartic forms we study and the ones from \cite{BR21} are essentially different. 
For instance, 
the subspace of even symmetric quartics has empty interior in the subspace of even quartics.
Thus the methods we use are essentially different to the methods from \cite{BR21}.\looseness=-1

Specializing Blekherman's techniques \cite{Ble06} 
(see Remark \ref{subsec-comparison-Blek}),
in \cite{KMSZ19} two types of cones of linear maps between matrix spaces were compared. Namely, the larger cone of positive maps and the smaller cone of completely positive maps. Using the correspondence
analogous to
\eqref{correspondence} 
the problem is equivalent to comparing the cone of nonnegative biquadratic biforms with the smaller cone of biforms that are sums of squares of bilinear ones. 
The conclusion is similar to the case of all forms \cite{Ble06}:
as the number of variables goes to infinity,
the ratio between the volume radii goes to 0.
\end{remark}

\subsection{Reader's guide}
In Section \ref{prelim} we first translate the problem
of comparing cones from Definition \ref{def-cones} to the comparison of cones in even quartic forms (\S\ref{translation}). Then we discuss the impact of the choice of compact sections of cones on the comparison results (\S\ref{choice-hyperplane}). In \S\ref{prel-even-quartics}--\S\ref{the-differential-metric} we establish some preliminary results for even quartic forms and
in \S\ref{BSantalo}--\S\ref{RShepard} we
state some classical inequalities for the volume of a bounded convex set in $\RR^n$.
In Section \ref{volume-radii}
these results are then applied in a novel way, together with three additional lemmas, 
in the proofs of Theorems \ref{intro-vrad-of-our-sets}
and \ref{vrad-quartics} (see Theorem \ref{vrad-of-our-sets}).
Firstly, Lemma \ref{inclusion-of-sections}
establishes the dilation constants of the difference bodies (see \eqref{difference-body}) of $\widetilde{\Lf_{\cQ}}$ and $\widetilde{\Phi(\CP_n')}$, that contain $\widetilde{\Phi(\nn_{n}')}$.
Secondly,
Lemma \ref{dualities-diff-metric} states the self-duality of $\Phi(\nn_{n}')$ and the duality between $\Pos_{\cQ}$ and $\Lf_{\cQ}$ in the differential metric (see \eqref{the-differential-metric}). Thirdly,
Lemma \ref{Santalo-of-Lf} states that the origin is the Santal\'o point of $\widetilde{\Lf_{\cQ}}$.
\looseness=-1

\section{Preliminaries}
\label{prelim}

In this section we translate our main problem of comparing sizes of matrix cones to the problem of comparing sizes of cones in even quartic forms. Then we discuss the impact of the choice of the hyperplane, in which we will compare the sizes of the sections of the cones, on our main results. Furthermore, we establish some properties of even quartic forms needed in the proofs of our main results.
These properties are obtained by specializing the results from \cite{Ble04,Ble06} to our setting. We also recall some classical inequalities for the volume of a compact set.

\subsection{Cones from quartic forms under the correspondence $\Phi$}
\label{translation}

Let $\cQ$, $\Phi$ be as in \eqref{101222-1925}, \eqref{correspondence}, respectively. 
Under the correspondence $\Phi$ we have the following bijections:
\begin{enumerate}
\item $\coPos_n$ corresponds to nonnegative forms $\Pos_\cQ$ from $\cQ$ (see \eqref{nonnegative-quartics}).
\item $\psd_n$ corresponds to sums of squares of quadratic forms in the variables $x_1^2,\ldots,x_n^2$:
\begin{equation*}
    \psd_\cQ:=
            \Big\{
                f\in \cQ\colon 
                f=\sum_{i} f_i^2\quad 
                \text{for some } f_i(\x)=\sum_{j=1}^n f_{j}^{(i)}x_j^2\in\RR[\x]_2
            \Big\},
\end{equation*}
\item $\nn_n$ corresponds to  forms in $\cQ$ with nonnegative coefficients:
\begin{equation*}
    \label{101222-2002}
    \nn_\cQ:=
            \Big\{
                 f\in \cQ\colon 
                f(\x)=\sum_{i,j=1}^n a_{ij}x_i^2x_j^2\quad
                \text{where }a_{ij}\geq 0\text{ for }i,j=1,\ldots,n 
            \Big\}.
\end{equation*}
\item   
    $\SPN_n$ corresponds to  forms
    $\SPN_\cQ$ in $\cQ$ that are sums of 
    a form from $\psd_\cQ$ and a form from $\nn_\cQ$,
    i.e., $\SPN_\cQ:=\psd_\cQ+\nn_\cQ$.
    It turns out that $\Sos_\cQ=\SPN_\cQ$ 
        \cite[p.\ 63--64]{Par00},
    where 
    $\Sos_\cQ$ is as in \eqref{101222-2006}.
    \item   
    $\DNN_n$ corresponds to  forms $\DNN_\cQ$ in $\cQ$ that belong to the intersection of
    the cones $\psd_\cQ$ and $\nn_\cQ$, i.e.,
    $\DNN_\cQ:=\psd_\cQ\cap\nn_\cQ$.
    \item   
    $\cp_n$ corresponds to  sums of squares of quadratic forms in variables 
		$x_1^2,\ldots,x_n^2$ 
	with nonnegative coefficients:
\begin{equation*}
    \label{151222-1322}
    \cp_\cQ:=
            \Big\{
                f\in \cQ\colon 
                f=\sum_{i=1} f_i^2\;\;
                \text{for some } 
                f_i(\x)=\sum_{j=1}^n f_{j}^{(i)}x_j^2\in\RR[\x]_2
                \text{ with }
                f_{j}^{(i)}\geq 0 \text{ for each }i,j
            \Big\}.
\end{equation*}
\end{enumerate}
Let $\cC$ be the set of cones we are interested in:
    \begin{equation}
        \label{cones-of-forms-studied}
            \cC:=\{\Pos_\cQ,{\Sos_\cQ=\SPN_\cQ},\nn_\cQ,\psd_\cQ,\DNN_\cQ,\Lf_\cQ,\CP_\cQ\},
    \end{equation}
where $\Lf_{\cQ}$ is as in \eqref{linear-forms-projections}.
We will estimate the gap between these cones of forms by estimating the volumes of compact sections obtained by intersecting each with a suitably chosen affine hyperplane.
Let $S^{n-1}$ be the unit sphere in $\RR^n$ and
$\sigma$ the 
rotation invariant probability measure
on $S^{n-1}$.
The natural $L^2$ inner product (resp.\ $L^2$ norm) on $\RR[\x]_4$ is given by 
\begin{equation}
\label{L2-inner-product}
\langle f,g\rangle=\int_{S^{n-1}}fg\; \dd\sigma\qquad \Big(\text{resp.}\;
        \left\| f \right\|_2^2= \int_{S^{n-1}} f^2\ \dd\sigma.\Big)
\end{equation}

\begin{remark}
\label{ort-proj}
    The projection $\pr_\cQ$, defined by \eqref{proj-Q}, is the orthogonal projection onto $\cQ$
    w.r.t.\ the $L^2$ inner product. This follows by noticing that
    $\int_{S^{n-1}}x_{k_1}\cdots x_{k_8}\dd\sigma$,
    where $k_1,\ldots,k_8\in \{1,\ldots,n\}$,
    is zero if there is $j\in \{1,\ldots,n\}$ which appears an odd number of times among the indices 
    $k_1,\ldots,k_8$ \cite[Lemma 8]{Bar02}.
\end{remark}

Let $\cL$ be the {affine} hyperplane of forms from $\RR[\x]_{4}$ of average 1 on $S^{n-1}$, i.e.,
\begin{equation}
\label{hyperplane-forms}
    \cL=
    \left\{ f\in \RR[\x]_{4}\colon \int_{S^{n-1}} f\; \dd\sigma=1 \right\}.
\end{equation}
For every $K\in \cC$ let
    $K'$
be its intersection with $\cL$, that is,
\begin{equation}
\label{compact-sections-forms}
    K'=K\cap \cL.
\end{equation}
Note that for all $K\in\cC$ the section $K'$ 
is a convex, compact, full-dimensional set
in the finite-dimensional {affine} hyperplane $\cL$.

\begin{remark}
\label{boundedness-v2}
As argument for boundedness of the largest set $\Pos_\cQ'$ is as follows.
Let $f(\x):=\sum_{i,j=1}^n a_{ij}x_i^2x_j^2\in \Pos_{\cQ}'$.
Then $\int_{S^{n-1}} f\dd\sigma=1$ and the equalities \eqref{integral-of-x4}, \eqref{integral-of-x2y2} below, imply 
that $3\Big(\sum_{i=1}^n a_{ii}\Big)+\sum_{i\neq j} a_{ij}=n(n+2).$
By \eqref{correspondence-quadratic}, $f(\x)=p_A(\x)=\x^TA\x$, where $A=[a_{ij}]_{i,j=1}^n$.
Since $A\in \COP'_n$, from here on the argument for boundedness of all coefficients $a_{ij}$ in the absolute value 
is the same as in Remark \ref{boundedness-sections}.
\end{remark}

For technical reasons we translate every section $K'$ for $K\in\cC$ by subtracting the polynomial
$(x_1^2+\ldots+x_n^2)^{2}$:
$$
    \widetilde K:=K'-(x_1^2+\ldots+x_n^2)^{2}=
    \left\{ f\in \cQ\colon  
        f+(x_1^2+\ldots+x_n^2)^{2}\in K'\right\}.
$$
Let $\cM$  be the hyperplane of forms from $\cQ$ with average 0 on $S^{n-1}$, i.e., 
\begin{equation} 
    \label{M-def}
	\cM=
    \left\{ f\in \cQ\colon \int_{S^{n-1}} f\; \dd\sigma=0 \right\}. 
\end{equation}
Notice that for every $K\in \cC$,
$$
    \widetilde{K}\subseteq \cM.
$$
With respect to the $L^2$ inner product $\cM$ is a subspace of $\cQ$ of dimension $\dim\cM=\frac{n(n+1)}{2}-1$ 
and so it is isomorphic to $\RR^{\dim{\cM}}$ as a Hilbert space.
Let $S_{\cM}$, $B_{\cM}$ be the unit sphere and the unit ball in $\cM$, respectively. 
Let $\psi:\RR^{\dim\cM}\to \cM$ be a unitary isomorphism and $\psi_{\ast}\mu$ the pushforward of the
Lebesgue measure $\mu$ on $\RR^{\dim\cM}$ to $\cM$, i.e., $\psi_\ast\mu(E):=\mu(\psi^{-1}(E))$ for every
Borel measurable set $E\subseteq \cM$.

\begin{remark}
\label{correspondence-discussion}
Assume the notation form \S\ref{subsec-main-results}.
In this remark we give a concrete presentation of the inner product $\langle \cdot,\cdot\rangle_{\Phi}$,
making $\Phi$ (see \eqref{correspondence}) a Hilbert space isomorphism.

Let $E_{ij}$ be the usual $n\times n$ matrix units, i.e., $E_{ij}$ has only one nonzero entry at the position $(i,j)$, which is equal to 1.
For $\Phi$ to be a Hilbert space isomorphism, 
$$\langle E_{ij}+E_{ji},E_{k\ell}+E_{\ell k}\rangle_{\Phi}=
\langle \Phi(E_{ij}+E_{ji}),\Phi(E_{k\ell}+E_{\ell k})\rangle
=
\int_{S^{n-1}} 4x_i^2x_j^2x_k^2x_{\ell}^2\;\dd\sigma
$$
should hold for each $i,j,k,\ell$.
Let $c_n=\frac{1}{n(n+2)(n+4)(n+6)}$.
Using \cite[Lemma 8]{Bar02}, we have that
\begin{align*}
    \langle E_{ii},E_{jj} \rangle_{\Phi}
    &:=
    \left\{
    \begin{array}{rl}
    105\cdot c_n,&   \text{if }i=j,\\[0.3em]
    9\cdot c_n,&   \text{if }i\neq j,
    \end{array}
    \right.\\[0.3em]
    \langle E_{ii},E_{j\ell}+E_{\ell j} \rangle_{\Phi}
    &:=
    \left\{
    \begin{array}{rl}
    30\cdot c_n,&   \text{if }j\neq \ell \text{ and }(i=j \text{ or } i=\ell),\\[0.3em]
    6\cdot c_n,&   \text{if } i,j,\ell\text{ are pairwise different},
    \end{array}
    \right.\\[0.3em]
    \langle E_{ij}+E_{ji},E_{k\ell}+E_{\ell k} \rangle_{\Phi}
    &:=
    \left\{
    \begin{array}{rl}
    36\cdot c_n,&   \text{if }i\neq j, k\neq \ell\text{ and among }i,j,k,\ell\text{ there are}\\ 
    &\text{exactly two pairs of equal indices},
        \\[0.3em]
    12\cdot c_n,&   \text{if }i\neq j, k\neq \ell \text{ and among }i,j,k,\ell\text{ there are}\\
    &\text{exactly two equal indices},\\[0.3em]
    4\cdot c_n,& \text{if }
        i,j,k,\ell \text{ are pairwise different}.
    \end{array}
    \right..
\end{align*}
\smallskip

Let $\mathcal D=\{\SPN,\nn,\psd,\DNN,\CP\}$.
Under the correspondence given by $\Phi$, we have the following equalities:
\begin{align*}
\Phi(\COP_{n})&=\Pos_{\cQ},\qquad 
\Phi(K_{n})=K_{\cQ}\;\;\text{for }K\in \mathcal D,\qquad
\Phi(\mathbf{1}_{n,n})=(x_1^2+\ldots+x_n^2)^2.
\end{align*}
Thus comparing the cones from Definition \ref{def-cones} 
can be done by passing to the corresponding cones in $\cQ$.
Intersecting the cones in $\cQ$ corresponds to intersecting the cones in $\sym_n$ by $\Phi^{-1}(\cL)$. By Lemma \ref{090123-1550} below, we have that
$\Phi^{-1}(\cL)=\cL_{\sym_n}$ and $\Phi^{-1}(\cM)=\cM_{\sym_n}$.
Thus,
\begin{align*}
\Phi(\widetilde{\COP_{n}})&=\widetilde{\Pos_{\cQ}},\qquad 
\Phi(\widetilde{K_{n}})=\widetilde{K_{\cQ}}\;\;\text{for }K\in \mathcal D.
\end{align*}
\end{remark}

\subsection{Discussion of the choice of the hyperplane $\cL$ and its impact on our main results}
\label{choice-hyperplane}
To compare the sizes of two cones $K_1,K_2\subseteq \RR^m$ it is natural to choose a compact set $C$ and compare the sizes of the intersections $K_1\cap C$ and $K_2\cap C$.
The choice of the set $C$ is not arbitrary, since it can have a large affect on the ratio 
$\frac{\size(K_1\;\cap \;C )}{\size(K_2\;\cap\; C)}$. Namely, compare the extreme situations where $C\subset K_1\cap K_2$ (the ratio is 1) or 
$C\subset K_2\setminus K_1$ (the ratio is 0).
Clearly, if we are interested in the sizes of the cones in a specific region $C$ in the space, then the ratio of sizes of the intersections with $C$ is a proper measure. However, for a general comparison of sizes it is more appropriate to choose $C$ that is symmetric enough,
so that it captures the picture of the whole space as uniformly as possible.
Such a suitable choice is to take the unit ball in some metric as seen on Figure \ref{fig:intersection-ball}.
\begin{center}
\begin{figure}[!h]
\begin{tikzpicture}[tdplot_main_coords]
  \coordinate (O) at (0,0,0);
  \draw[gray!60!white, ->] (-6,0,0) -- (6,0,0);
  \draw[gray!60!white, ->] (0,-6,0) -- (0,6,0);
  \draw[gray!60!white, ->] (O) -- (0,0,5);
   \draw [black!60!white](0,0,0) circle (1.5cm);
    \shade[ball color=black!10!white,opacity=0.20] (0,0,0) circle (1.5cm);
    \draw[black!60!white] (-1.06,-1.06,0) arc (180:360:1.5cm and 0.5cm);
 \draw[black!20!white] (-1.06,-1.06,0) arc (180:0:1.5cm and 0.5cm);
 
  \conebackred[surface]{3}{3}{10}
  \coneback[surface]{4}{2}{15}
  
  \conefront[surface]{4}{2}{15}
  \conefrontred[surface]{3}{3}{10}

\node[right=1cm, color = red!40!white] at (current bounding box.east) {
    {\rule{2em}{2pt}} \textcolor{black}{$K_1$}
};

\node[below=1cm, left=0cm, color=blue!40!white] at (current bounding box.east) {
    {\rule{2em}{2pt}} \textcolor{black}{$K_2$}
};

\node[below=2cm, left=-1cm, color=black!20!white] at (current bounding box.east) {
    {\rule{2em}{2pt}} \textcolor{black}{unit ball}
};

    \end{tikzpicture}
\label{figure1}\caption{Intersection of the cones with the unit ball in some metric.}
  \label{fig:intersection-ball}
\end{figure}
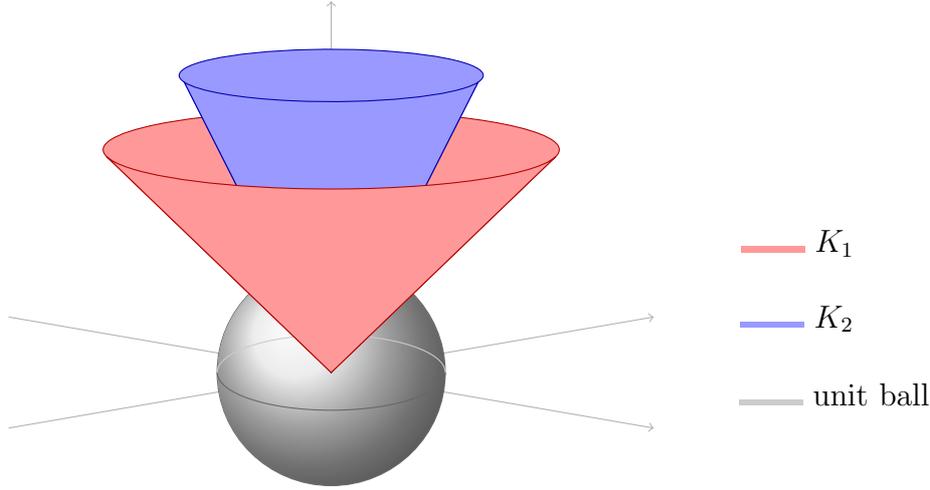
\end{center}

In our setting we equipped the vector space $\cQ$ with the $L^2$ norm and thus a natural choice for $C$ is the unit ball $B_2=\{f\in \cQ\colon \|f\|_2\leq 1\}$ in the $L^2$ norm.
Writing down the elements in $B_2$ explicitly, one ends up with a quadratic inequality
$\sum_{i,j,k,\ell} c_{ijk\ell} a_{ij}a_{k\ell}\leq 1$ on the coefficients $a_{ij}$ of $f(\x)=\sum_{i,j=1}^n a_{ij} x_i^2x_j^2$, where $c_{ijk\ell}$ are some constants.
Since the inequality is quadratic in $a_{ij}$, tight volume estimates of $K\cap B_2$ for $K\in \cC$ (see \eqref{cones-of-forms-studied}),
such that the gap between the lower and the upper bound follows the same asymptotics as $n\to \infty$, cannot be easily derived. Replacing the $L^2$ norm with a $L^p$ norm for $p>2$ brings conditions on $a_{ij}$, which are even more difficult to work with, since they are of degree $p$ in $a_{ij}$.
Replacing the $L^2$ norm with the $L^1$ norm $\|\cdot\|_1$, the unit ball is
\begin{equation}
\label{unit-ball-L1}
B_1
=\{f\in \cQ\colon \|f\|_1\leq 1\}
=\Big\{\sum_{i,j=1}^n a_{ij} x_i^2x_j^2\in \cQ\colon 
\int_{S^{n-1}} |\sum_{i,j=1}^n a_{ij} x_i^2x_j^2|\leq 1
\Big\}.
\end{equation}
Due to the presence of the absolute value in the integrand of \eqref{unit-ball-L1},
it is still not easy to derive tight volume estimates. 
However, there is a norm, in which ratios can be easily derived, i.e., the $\ell^1$ norm $\|\cdot\|_1'$ of the vector of coefficients of the polynomial. The unit ball $B_1'$
in this norm is 
$$
B_1'
=\{f\in \cQ\colon \|f\|_1'\leq 1\}
=\Big\{\sum_{i,j=1}^n a_{ij} x_i^2x_j^2\in \cQ\colon 
\sum_{i,j=1}^n |a_{ij} |\leq 1
\Big\}.
$$
By a simple argument using the Rogers-Shepard inequality (Theorem \ref{RS} below) 
it is easy to see that 
\begin{equation}
\label{vrad-ratio-ell-1}
\frac{1}{12}\leq\frac{\vrad(\CP_{\cQ}\;\cap\; B_1')}{\vrad (\Pos_{\cQ}\;\cap\; B_1')}\leq 1.
\end{equation}
Indeed, let
$\displaystyle f=\sum_{1\leq i\leq j\leq n} a_{ij}x_i^2x_j^2\in B_1'$ 
and define 
$$b_{i}:=a_{ii}-\frac{1}{2}\big(\sum_{k=1}^{i-1}a_{ki}+\sum_{j=i+1}^n a_{ij}\big)
\quad \text{for each }i.$$
Then $f$ can be written as
\begin{align}
\label{f-in-ell1-norm}
\begin{split}
&    f(\x)
=\sum_{1\leq i< j \leq n}
\frac{a_{ij}}{2}(x_i^2+x_j^2)^2
+
\sum_{1\leq i \leq n} b_i x_i^4\\
&=
\underbrace{\sum_{\substack{1\leq i< j \leq n,\\[0.2em] a_{ij}\geq 0}}
\frac{a_{ij}}{2}(x_i^2+x_j^2)^2
+
\sum_{\substack{1\leq i \leq n,\\[0.2em] 
b_i\geq 0}} b_i x_i^4}_{f_1(\x)}
-\underbrace{\Big(\sum_{\substack{1\leq i< j \leq n,\\[0.2em] a_{ij}< 0}}
\frac{|a_{ij}|}{2}(x_i^2+x_j^2)^2
+
\sum_{\substack{1\leq i \leq n,\\[0.2em] b_i<0}}
|b_i| x_i^4\Big)}_{f_2(\x)}.
\end{split}
\end{align}
Note that $f_1(\x),f_2(\x)\in \CP_{\cQ}$,
$\|f_1(x)\|_1'\leq 3$, $\|f_2(x)\|_1'\leq 3$
and hence, 
$$B_1'\subseteq 3\big((CP_Q\cap B_1')-(CP_Q\cap B_1')\big).$$
This and Theorem \ref{RS} below, imply the left-hand side inequality in \eqref{vrad-ratio-ell-1},
while the right-hand side one is clear.

From \eqref{vrad-ratio-ell-1} we can conclude that in the unit ball of the $\ell^1$ norm
the sizes of all cones from \eqref{cones-of-forms-studied} 
follow the same asymptotics up to absolute constants, which lie on the interval 
$[\frac{1}{12},1]$.
However, it does not follow that when replacing $B_1'$ with a ball in another metric, e.g., $B_2$, the asymptotics of the sizes of the sections would still differ only up to absolute constants. Indeed, extending the definition of the $\ell^1$ norm to all forms in $\RR[\x]_4$ and by an analogous decomposition as in \eqref{f-in-ell1-norm} 
using the equalities
$$
x_{i}x_jx_kx_\ell
=
\frac{1}{2}(x_ix_j+x_{k}x_{\ell})^2-\frac{1}{2} x_i^2x_j^2-\frac{1}{2} x_k^2x_\ell^2,
$$
it follows that in the unit ball of the $\ell^1$ norm the sizes of the cone of positive quartics and the cone of sums of squares quartics follow the same asymptotics up to an absolute constant.
But this is in sharp contrast with the results of Blekherman \cite{Ble06} (see Remark \ref{subsec-comparison-Blek} above).
However, in \cite{Ble06}, the cones of forms are compared in the $L^2$ metric. 

Next we discuss how to obtain the volume estimates of the cones for the unit ball of the $L^2$ norm,
which is difficult to establish directly, as explained in the second paragraph above. 
To get around this issue it is natural to reduce the problem to estimating volumes of intersections of the cones with a hyperplane, having a chosen unit vector for its normal. There are two requirements for the proper choice of the hyperplane, i.e., the intersections with the cones must be compact and the hyperplane must intersect all cones ``fairly''. If the cones are in general position, not having any common line of symmetry, there might not exist such a hyperplane. However, in case the cones share a unique line of symmetry, it is natural to take the hyperplane with the normal being this line of symmetry (see Figure \ref{fig:intersection-hyperplanes}).
The cone of all forms $\RR[\x]_4$, the cone of positive forms and the cone of sums of squares forms are all easily seen to be invariant under the action of the orthogonal group $O(n)$ defined by rotating the coordinates:
\begin{equation}
    \label{action}
    O\cdot f(\x):=f(O^{-1}\x)\quad \text{for }O\in O(n).
\end{equation} Moreover, the only fixed points for this action are the polynomials $\alpha (x_1^2+\ldots+x_n^2)^2$, $\alpha\in \RR$. Hence, choosing the hyperplane with a normal $(x_1^2+\ldots+x_n^2)^2$
is the best choice to compare sizes of the cones.
The results in \cite{Ble06} also compare the cones with this choice.\looseness=-1

Referring to the previous paragraph, a natural question is, whether also the difference between our cones of even forms becomes visible, when the unit ball in the $\ell^1$ norm is replaced by the unit ball in the $L^2$ norm.  
Even forms are not invariant under the above action \eqref{action} of $O(n)$, but only under its subgroup $S(n)$ of all permutations of the coordinates. There are more fixed points under the action of $S(n)$, i.e., every point $\alpha \sum_i x_i^4+\beta \sum_{i<j}x_i^2x_j^2$, $\alpha,\beta\in \RR$ is a fixed point. 
However, $r(\x)=(x_1^2+\ldots+x_n^2)^2$ is still the most fair choice among them, since it is the only one that extends to the fair choice for larger cones in $\RR[\x]_4$. Choosing $r(\x)$ for the normal, we get precisely the hyperplane $\cL$ (see \eqref{hyperplane-forms}), since 
$1=\langle f,r\rangle=\int_{S^{n-1}} fr\;\dd\sigma=\int_{S^{n-1}} f\;\dd\sigma$.


\begin{center}
\begin{figure}[!h]

 \begin{tikzpicture}
\coordinate (O) at (0,0);

\fill[blue!40!white,opacity=.4] (-4,5) -- (O) -- (4,5);
\draw[ultra thick, blue!70!white,opacity=.6] (-4,5) -- (O) -- (4,5);
\fill[red!40!white,opacity=.4] (-2,5) -- (O) -- (2,5);
\draw[ultra thick, red!70!white,opacity=.6] (-2,5) -- (O) -- (2,5);

\draw[ultra thick, green!40!black, opacity=.6] (-4,3) -- (4,3);
\draw[ultra thick, orange!70!black, opacity=.6] (-4,5.2) -- (.6,-.2);

\draw[gray!70!black, ->] (-6,0) -- (6,0);
\draw[gray!70!black, ->] (0,0) -- (0,5.2);

\node[right=1cm, above=.7cm, color = blue!40!white] at (current bounding box.east) {
    {\rule{2em}{2pt}} \textcolor{black}{$\COP$}
};

\node[below=1cm, above=1cm, left=.3cm, color=red!40!white] at (current bounding box.east) {
    {\rule{2em}{2pt}} \textcolor{black}{$\Lf$}
};

\node[below=2cm, above=1cm, left=-0.84cm, color=green!70!black] at (current bounding box.east) {
    {\rule{2em}{2pt}} \textcolor{black}{fair plane}
};
\node[below=3cm, above=1cm, left=-.47cm, color=orange!70!black] at (current bounding box.east) {
    {\rule{2em}{2pt}} \textcolor{black}{unfair plane}
};

\end{tikzpicture}
  \label{figure2}\caption{Intersections of the cones with a fair and an unfair hyperplane.}
  \label{fig:intersection-hyperplanes}
\end{figure}
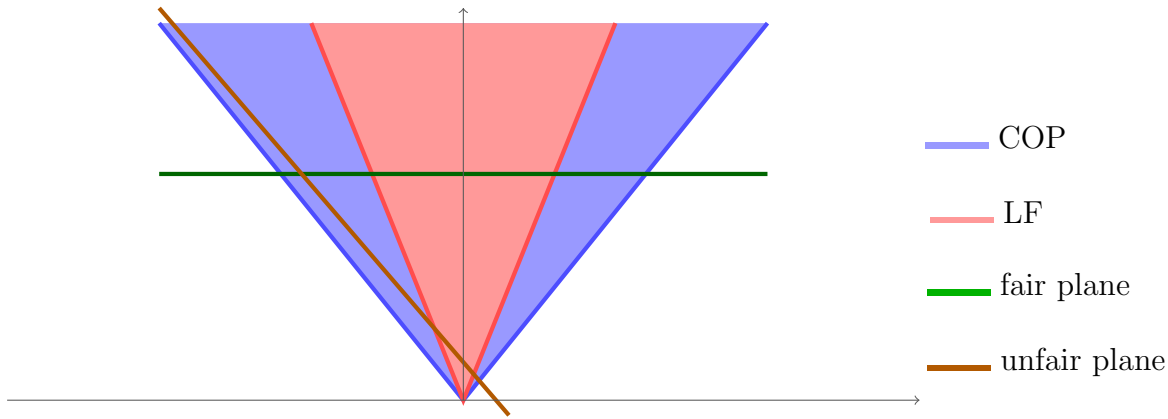
\end{center}

\begin{remark}[Discussion of the ratio of volumes versus the ratio of volume radii]
\label{volumes-vs-vrads}
Let $\pr_{n}:\RR^{n}\to \RR^{n-1}$ be the projection of $\RR^n$ to the first $n-1$ coordinates, i.e., 
$\pr_{n}(x_{1},\ldots,x_n)=(x_1,\ldots,x_{n-1})$.
Let $(K_n)_n$ and $(L_n)_n$ be two sequences of cones, 
where $K_n, L_n\subseteq \RR^{n}$ and 
$\pr_{n}(K_n)=K_{n-1}$, $\pr_{n}(L_n)=L_{n-1}$ for each $n$.
Let $(\cH_n)_n$ be a sequence of affine hyperplanes $\cH_n\subset \RR^n$, such that
the intersections $K_n':=K_n\cap \cH_n$ and $L_n':=L_n\cap \cH_n$ are compact.
When comparing asymptotic sizes of the cones $K_n$ and $L_n$, as $n\to \infty$, based on the sequences
$(K'_n)_n$ and $(L_n')_n$,
it is important that comparing $K_{n-1}'$ and $L_{n-1}'$ is equivalent to comparing
$\pr_{n}(K'_n)$ and $\pr_n(L'_{n})$ for each $n$.
Otherwise the choice of the compact section in each dimension could have a large impact on the size estimates (see also the discussion in \S\ref{choice-hyperplane} above). 
Let now $(c_n)_n$ be a sequence of positive constants, such that 
$\Vol L_n'=(c_n)^{n-1}\cdot\Vol K_n'$.
The constant $c_n$ is the dilation coefficient of $K_n'$ 
such that the sizes of $L_n'$ and $c_nK_n'$ are the same.
Studying the impact of the dimension on the difference of sizes of $K_n'$ and $L_n'$ is all contained in the sequence $(c_n)_n$. If $c_n\equiv c$ for each $n$, then for our purposes the dimension does not have any impact on the size difference. 
So computing 
the normalized volume
ratio $\Big(\frac{\Vol L_n'}{\Vol K_n'}\Big)^{1/(n-1)}=c_n$ is the relevant size difference information for us. Choosing a reference sequence $(B_n)_n$ of unit balls in every affine hyperplane $\cH_n$ (once the origin is fixed and the metric comes form the inner product on $\cH_n$) and comparing a sequence $(K_n)_n$ with $(B_n)_n$, leads to the notion of $\vrad(K_n)$. This is a sequence of dilation constants of the sequence $(B_n)_n$, such that the volumes of $c_nB_n$ and $K_n$ are the same for every $n$.

Note that comparing the sequences of volume radii of two sequences $(K'_n)_n$ and $(L'_n)_n$, where $K_n'\subseteq L_n'$ for each $n$, leads to the information on the dimension effect on the size ratio of the sequences. However, if we fix a level $n$ and consider the question of probability that a given point, obtained by random sampling according to the uniform distribution on the larger section $L_n'$, also lies in the smaller section $K_n'$, we should not neglect the effect of the dimension. 
But even for this problem all the information is stored in the sequence of volume radii ratios.
\end{remark}

\subsection{Membership in $\cM$, $\cL$ and the space of harmonic polynomials}
\label{prel-even-quartics}

Let 
    $$
    \Delta
    :=
        \frac{\partial^2}{\partial x_1^2}
        +\ldots+
        \frac{\partial^2}{\partial x_n^2}
    $$
be the Laplace operator.
A form $f\in \RR[\x]_{k}$ is called \textbf{harmonic} if $\Delta(f)=0$. 
We denote by $\cH_{4}$ the space of all harmonic forms of degree $4$.

The following lemma characterizes 
the membership of $f\in \cQ$ in $\cL$, $\cM$, $\cH_4$
in terms of its coefficients.

\begin{lemma}
   \label{090123-1550}
    Let 
        $\displaystyle f(\x)=\sum_{i,j=1}^n a_{ij}x_i^2x_j^2\in \cQ$.
    The following statements hold:
\begin{enumerate}[\rm(1)]
    \item 
        \label{prop-of-f-in-Q-part0}
            $
            \displaystyle f\in \cL
            \quad \Leftrightarrow\quad
            3\Big(\sum_{i=1}^n a_{ii}\Big)+\sum_{i\neq j} a_{ij}=n(n+2).
            $
        \medskip
    \item 
        \label{prop-of-f-in-Q-part1}
            $
            \displaystyle f\in \cM
            \quad \Leftrightarrow\quad
            3\Big(\sum_{i=1}^n a_{ii}\Big)+\sum_{i\neq j} a_{ij}=0.
            $
    \medskip
    \item
        \label{prop-of-f-in-Q-part2}
            $
            \displaystyle f\in \cH_4
            \quad \Leftrightarrow\quad
            a_{ii}=
            -\frac{1}{6}
            \sum_{
                \substack{j=1,\ldots,n,\\ j\neq i}
                }
                (a_{ij}+a_{ji})\quad \text{for }
            i=1,\ldots,n.
            $
    \medskip
    \item 
        \label{prop-of-f-in-Q-part3}
        $
        \cH_4\subseteq \cM.
        $
    \medskip
    \item 
        \label{prop-of-f-in-Q-part4}
        $
        \dim (\cH_4\cap \cQ)=\frac{n(n-1)}{2}.
        $
\end{enumerate}
\end{lemma}
\begin{proof}
    	\ref{prop-of-f-in-Q-part0}
    and
     	\ref{prop-of-f-in-Q-part1}
    follow by using \cite[Lemma 8]{Bar02} to check that
    \begin{align}
	 \label{integral-of-x4}
        \int_{S^{n-1}} x_i^4\ \dd\sigma
            &=\frac{3}{n(n+2)}\quad \text{for }i=1,\ldots,n,\\
	\label{integral-of-x2y2}
        \int_{S^{n-1}} x_i^2x_j^2\ \dd\sigma
            &=\frac{1}{n(n+2)}\quad  \text{for }1\leq i\neq j\leq n.
    \end{align}

    \ref{prop-of-f-in-Q-part2} follows from the following computation:
    \begin{align*}
         \Delta\big(\sum_{i,j=1}^n a_{ij}x_i^2x_j^2\big)
         &=
            \sum_{i=1}^n (12 a_{ii}x_i^2)
            +
            \sum_{i\neq j}(2a_{ij}(x_i^2+x_j^2))\\
         &=
         \sum_{i=1}^n 
            \Big(
                x_i^2 
                \big(12a_{ii}+
                    \sum_{\substack{j=1,\ldots,n\\ j\neq i}}
                    2(a_{ij}+a_{ji})
                \big)
            \Big).
    \end{align*}
    Hence, 
    $$
        \Delta\big(\sum_{i,j=1}^n a_{ij}x_i^2x_j^2\big)=0
        \quad \Leftrightarrow \quad
        12a_{ii}+\sum_{\substack{j=1,\ldots,n\\ j\neq i}}
                    2(a_{ij}+a_{ji})=0
                    \quad \forall i=1,\ldots,n,
    $$
    which proves \ref{prop-of-f-in-Q-part2}.

    \ref{prop-of-f-in-Q-part3} follows easily from
    \ref{prop-of-f-in-Q-part2} 
    and 
    \ref{prop-of-f-in-Q-part1},
    while the following computation using \ref{prop-of-f-in-Q-part2},
    $$
    \dim(\cH_4\cap \cQ)=\dim \cQ-n=\frac{n(n+1)}{2}-n=\frac{n(n-1)}{2},
    $$
   proves \ref{prop-of-f-in-Q-part4}. 
\end{proof}

\subsection{The differential metric}
    \label{the-differential-metric}

In this subsection we follow \cite[Section 5]{Ble06}.
For a form 
    $$
    f(\x)
    =\sum_{1\leq i,j,k,\ell\leq n}
        a_{ijk\ell}x_ix_jx_kx_\ell
    \in \RR[\x]_4
    $$
the \textbf{differential operator} $D_f:\RR[\x]_4\to \RR$
is defined by
    $$
    D_f(g)
    =\sum_{1\leq i,j,k,\ell\leq n}
        a_{ijk\ell}
        \frac{\partial^4g}{\partial x_i\partial x_j\partial x_k\partial x_\ell}.
    $$
The \textbf{differential metric} on $\RR[\x]_4$ is given by
    $$\langle f,g\rangle_d=D_f(g).$$
For a point $v=(v_1,\ldots,v_n)\in S^{n-1}$, we denote by $v^4$
the fourth power of a linear form:
    $$
    v^4:=(v_1x_1+\ldots+v_nx_n)^4.
    $$
For us the following operator $T:\RR[\x]_4\to \RR[\x]_4$ will be important:
    $$
        (Tf)(\x)=\int_{S^{n-1}}f(v)v^4\dd\sigma(v).
    $$
Let
\begin{align*}
    {\cH_0}
    &:=
    \Big\{c\big(\sum_{i=1}^n x_i^2\big)^{2}\colon c\in \RR
    \Big\},\\
    {\cH_2}
    &:=
    \Big\{
    g\in \RR[\x]_4\colon    
        g=\big(\sum_{i=1}^n x_i^2\big)\cdot h
	\quad
        \text{for some harmonic form } h\in \RR[\x]_{2}
    \Big\}.
\end{align*}
It is well-known that every $f\in \RR[\x]_4$ can be uniquely written as a sum $f=f_0+f_1+f_2$, where $f_i\in {\cH_{2i}}$, $i=0,1,2$ \cite[Theorem 2.1]{Ble04}.
We denote by 
    $\ell_i:\RR[\x]_4\to \cH_{2i}$,
$i=0,1,2$,
the corresponding projection operators, i.e., 
$\ell_i(f)=f_i.$

We collect some properties of $T$ in the next lemma.

\begin{lemma}
    \label{T-properties}
    The operator $T$ satisfies:
    \begin{enumerate}[\rm (1)]
        \item\label{T-prop-pt4}  
            $\langle Tf,g\rangle_d=4!\langle f,g\rangle$ for every $f,g\in \RR[\x]_4$.
            \medskip
        \item\label{T-prop-pt1}
            $\cQ$ is an invariant subspace of $T$.
            \medskip
        \item\label{T-prop-pt2}
            $\displaystyle T\Big(\big(\sum_{i=1}^n x_i^2\big)^2\Big)=\frac{3}{n(n+2)}\big(\sum_{i=1}^n x_i^2\big)^2.$
            \medskip
        \item\label{T-prop-pt3}
		$
            \displaystyle \frac{n(n+2)}{3}T(f)=    
            \ell_0(f)+\frac{4}{n+4}\ell_1(f)+\frac{8}{(n+4)(n+6)}\ell_2(f)
            $
           for every $f\in \RR[\x]_4$.
         \medskip
         \item\label{T-prop-pt5}
            The restriction $T|_{\cQ}:\cQ\to\cQ$ of $T$ to $\cQ$ is bijective.
    \end{enumerate}
\end{lemma}

\begin{proof}
    \ref{T-prop-pt4} is \cite[Lemma 5.1]{Ble06}.

    Next, we establish \ref{T-prop-pt1}.
    Let $f(\x)=\sum_{i,j=1}^n a_{ij}x_i^2x_j^2\in \cQ$.
    Then
    \begin{align*}
    (Tf)(\x)
        &=\int_{S^{n-1}}f(v)v^4\dd\sigma(v)
        =\int_{S^{n-1}}\left(\sum_{i,j=1}^na_{ij}v_i^2v_j^2\right)v^4\dd\sigma(v)\\
        &=\sum_{i,j=1}^n\sum_{k_1,k_2,k_3,k_4=1}^n
        \left(a_{ij}x_{k_1}x_{k_2}x_{k_3}x_{k_4}
        \int_{S^{n-1}}\left(v_i^2v_j^2v_{k_1}v_{k_2}v_{k_3}v_{k_4}\right)\dd\sigma(v)\right)
    \end{align*}
    By \cite[Lemma 8]{Bar02}, the integral of 
        $v_i^2v_j^2v_{k_1}v_{k_2}v_{k_3}v_{k_4}$
    over $S^{n-1}$ is nonzero iff all the exponents at the coordinates of $v$
    are even. Hence, either $k_1=k_2=k_3=k_4$ or the indices $k_i$ split into two equal pairs,
    i.e., $k_{i_1}=k_{i_2}$ and $k_{i_3}=k_{i_4}$, where $i_1,i_2,i_3,i_4$ is some permutation of 
    the indices $1,2,3,4$. In both cases $x_{k_1}x_{k_2}x_{k_3}x_{k_4}\in \cQ$ and hence 
    $Tf\in \cQ$.

    \ref{T-prop-pt2} is a special case (for $k=2$) of the first paragraph of the proof of \cite[Lemma 5.2]{Ble06} showing that 
    \begin{equation}
		\label{constant-c}
        T\Big(\big(\sum_{i=1}^n x_i^2\big)^2\Big)=c\big(\sum_{i=1}^n x_i^2\big)^2\quad  
    \text{for }
        c=\int_{S^{n-1}}x_1^4 \dd\sigma=\frac{3}{n(n+2)}.
    \end{equation}

    \ref{T-prop-pt3} follows by \cite[Lemma 7.4]{Ble04}, where it is shown that
    $$
        \frac{1}{c}T(f)=c_0\ell_0(f)+c_1\ell_1(f)+c_2\ell_2(f),
    $$
    with $c$ as in \eqref{constant-c} and
    $
    c_0=1,
    c_1=\frac{4}{n+4},
    c_2=\frac{8}{(n+4)(n+6)}.
    $

    It remains to prove \ref{T-prop-pt5}.
    Using the fact that
    $\RR[x]_4$ is a direct sum of the subspaces
    $\cH_0$, $\cH_2$ and $\cH_4$ \cite[Theorem 2.1]{Ble04}
    and \ref{T-prop-pt3}, 
	it follows that $T:\RR[x]_4\to \RR[x]_4$ is bijective.
	In particular, $T|_{\cQ}$ is injective and 
	since $T|_{\cQ}$ maps to the finite-dimensional $\cQ$ by \ref{T-prop-pt1},
	\ref{T-prop-pt5} follows.
\end{proof}

Let $\cL$ and $\cM$ be as in \eqref{hyperplane-forms} and \eqref{M-def}, respectively.
Given a full-dimensional cone $L\subseteq \cQ$
such that 
        $\big(\sum_{i=1}^n x_i^2\big)^2$
is in the interior of $L$ and $\int_{S^{n-1}} f \dd\sigma>0$ for every nonzero $f\in L$,
we define the sets 	
$$
	L'=L\cap \cL
	\qquad\text{and}\qquad
	\widetilde{L}
	=
	\left\{
	f\in \cM\colon f+\big(\sum_{i=1}^n x_i^2\big)^2\in L
	\right\}.
$$
Let $L^\ast$ and $L^{\ast}_d$ be the duals of $L$ in the $L^2$ metric and the differential metric,
respectively:
\begin{align*}
       L^\ast
	&=
        \left\{
            f\in \cQ\colon      
                \langle f,g\rangle\geq 0\quad\forall g\in L
        \right\},\\
       L^\ast_d
	&=
        \left\{
            f\in \cQ\colon      
                \langle f,g\rangle_d\geq 0\quad\forall g\in L
        \right\}.
\end{align*}
The following lemma is an analog of \cite[Lemma 5.2]{Ble06} for $\cQ$.

\begin{lemma}
    \label{vrad-int-vs-d-metric}
    Let $L$ be a full-dimensional cone in $\cQ$ such that 
        $\big(\sum_{i=1}^n x_i^2\big)^2$
    is in the interior of $L$ and $\int_{S^{n-1}} f \dd\sigma>0$ for every nonzero $f\in L$.
    Then:
    $$
        \frac{8}{(n+4)(n+6)}
        \leq 
        \Big(
            \frac
                {\Vol\widetilde{L^{\ast}_d}}
                {\Vol\widetilde{L^{\ast}}}
        \Big)^{1/\dim\cM}
        =
            \frac
                {\vrad\widetilde{L^{\ast}_d}}
                {\vrad\widetilde{L^{\ast}}}
        \leq
        \Big(\frac{8}{(n+4)(n+6)}\Big)^{1-\frac{2n-1}{n^2+n-1}}.
    $$
\end{lemma}

\begin{proof}
    We follow the proof of \cite[Lemma 5.2]{Ble06}.
    By \ref{T-prop-pt4} of Lemma \ref{T-properties},
    for every $f,g\in \cQ$ we have
        $$
            \langle f,g\rangle \geq 0 
            \quad \Leftrightarrow \quad
            \langle Tf,g\rangle_d \geq 0.
        $$
    Since by 
	\ref{T-prop-pt5} of Lemma \ref{T-properties}, $T$ maps $\cQ$ bijectively to $\cQ$, it follows that 
    \begin{equation}
        \label{T-on-sections}
            T(L^\ast)=L^{\ast}_d.
    \end{equation}
    Observe that 
        $\cM=\widetilde{\cH_2}\oplus\widetilde{\cH_4},$
    where $\widetilde{\cH_{2i}}=\cH_{2i}\cap \cM$, $i=1,2$.
    Since $\cH_2$, $\cH_4$ are invariant subspaces of $\frac{n(n+2)}{3}T$ by 
    \ref{T-prop-pt3} of Lemma \ref{T-properties},
    it follows that $\cM$ is also an invariant subspace of $\frac{n(n+2)}{3}T$,
    which in addition also fixes $\big(\sum_{i=1}^n x_i^2\big)^2$.
    This, together with \eqref{T-on-sections}, implies that 
    \begin{equation}
        \label{T-on-sections-2}
            \Big(\frac{n(n+2)}{3}T\Big)(\widetilde{L^\ast})=\widetilde{L^{\ast}_d}.
    \end{equation}
    Since by \ref{T-prop-pt3} of Lemma \ref{T-properties}, the operator 
    $\frac{n(n+2)}{3}T$ acts as a contraction on the subspaces $\widetilde{\cH_{2i}}$, $i=1,2$,
    with the smallest contraction coefficient 
        $\frac{8}{(n+4)(n+6)}$,
    this establishes the lower bound in the statement of the lemma.
    To get the upper bound observe that the largest contraction occurs in $\widetilde{\cH_4}$
    where
        $\dim\widetilde{\cH_4}=\frac{n(n-1)}{2}$  
    by \ref{prop-of-f-in-Q-part4} of 
    Lemma \ref{090123-1550}.
    Since 
        $$
            \frac
                {\dim\widetilde{\cH_4}}{\dim\cM}
            =\frac
                {\frac{n(n-1)}{2}}
                {\frac{n(n+1)-1}{2}}
            =\frac{n(n-1)}{n(n+1)-1}=1-\frac{2n-1}{n^2+n-1},
        $$
    this establishes the upper bound of the lemma.
\end{proof}

\subsection{Blaschke-Santal\'o inequality and its reverse}
\label{BSantalo}

Let $K$ be a bounded convex set in $\RR^n$ with origin in its interior and $\langle\cdot,\cdot\rangle$
the inner product on $\RR^n$. 
The \textbf{polar} $K^\circ$ of $K$ is defined by
\begin{equation*}
	K^{\circ}
	=
	\left\{
		y\in \RR^n\colon \langle x,y\rangle\leq 1\quad\forall x\in K
	\right\}.
\end{equation*}
For $z\in \RR^n$ let
\begin{equation*}
	K^{z}
	=
	\left\{
		y\in \RR^n\colon \langle x-z,y-z\rangle\leq 1\quad\forall x\in K
	\right\},
\end{equation*}
i.e., $z$ is translated to the origin.
We denote by $p(K)$ the \textbf{volume product}  of $K$:
\begin{equation}
	\label{volumic-product}
	p(K)
	=
	\inf\{
		\Vol(K)\Vol(K^z)\colon z\text{ is an interior point of } K
	\}.
\end{equation}
It is clear that $p(K)$ is affine invariant, i.e., for every affine linear {invertible} transformation $T:\RR^m\to\RR^m$
it holds that 
	$p(K)=p(TK)$.
It turns out that there is a unique point $z$ 
\begin{color}{black} \cite[p.\ 85]{MP90} \end{color}, where the infimum in \eqref{volumic-product} is attained. 
This point is called the \textbf{Santal\'o point} of $K$ and we denote it by $s(K)$.
The following upper bound holds for the volume product.

\begin{theorem}
    [{Blaschke-Santal\'o inequality,  \cite[p.\ 90]{MP90}}]
    \label{BS}
	For a bounded convex set $K$ in $\RR^n$ with a non-empty interior it holds that
	\begin{equation*}
	\label{BSr-original}
		\Vol(K)\Vol(K^{s(K)})\leq (\Vol(B))^2,
	\end{equation*}
    where $B$ is the unit ball w.r.t.\ the inner product $\langle\cdot,\cdot\rangle$ on $\RR^n$.
\end{theorem}

There is also a lower bound for the volume product.

\begin{theorem}
    [{Reverse Blaschke-Santal\'o inequality}]
    \label{revBS}
	For a bounded convex set $K$ in $\RR^n$ with the origin in its interior, it holds that
	\begin{equation}
	\label{BSr-reverse}
		4^{-n}\pi n(\Vol(B))^2< \Vol(K)\Vol(K^\circ),
	\end{equation}
    where 
    $B$ is the unit ball w.r.t.\ the inner product $\langle\cdot,\cdot\rangle$ on $\RR^n$.
\end{theorem}

{
Bourgain and Milman \cite[Corollary 6.1]{BM87} established a weaker version of Theorem \ref{revBS} with $4^{-n}\pi n$ replaced by $c^n$ for an absolute constant
 $c>0$  independent of the dimension $n$ and the set $K$.
Later, Kupperberg \cite[Corollary 1.8]{Kup08} proved that 
	$\Vol(K)\Vol(K^\circ)\geq 4^n\big(\frac{(n!)^2}{(2n!)}\big)^{2}(\Vol(B))^2$.
Since
	$4^n\big(\frac{(n!)^2}{(2n!)}\big)^{2}= \Big(2^n\big(\frac{(n!)^2}{(2n!)}\big)\Big)^2\geq 2^{-2n}=4^{-n}$,
 one immediately gets $4^{n}(\Vol(B))^2$ as a lower bound on $\Vol(K)\Vol(K^\circ)$. To obtain the lower bound from Theorem \ref{revBS}, we use the following asymptotically sharp estimate:
$$
4^n\big(\frac{(n!)^2}{(2n!)}\big)^{2}=4^n\big(\binom{2n}{n}\big)^{-2}> \big(2^{-n} \sqrt{\pi n}\big)^2 = 4^{-n} \pi n.
$$
Indeed, following an idea from \cite{UL-mo} one can estimate the central binomial coefficient as follows:
\begin{align}
\label{binom-coeff}
\begin{split}
  \binom{2n}{n} &= \frac{4^n}{\pi} \int_{-\pi/2}^{\pi/2} \cos^{2n} x\, dx \leq   
  \frac{4^n}{\pi} \int_{-\pi/2}^{\pi/2} e^{-nx^2} x\, dx \\
  &< \frac{4^n}{\pi} \int_{-\infty}^{\infty} e^{-nx^2} x\, dx = \frac{4^n}{\sqrt{\pi n}}.
\end{split}
\end{align}
The first equality can be obtained either by induction and integration by parts or by writing $\cos x = (e^{\mathfrak{i}x} + e^{-\mathfrak{i}x})/2,$ expanding $\cos^{2n}x$ by the binomial theorem and directly computing the integrals involving the exponential function.
For the first inequality we estimate $\cos x \leq e^{-x^2/2}$ on $|x|\leq \pi/2$ by noting that
$$
\log \cos x + \frac{x^2}{2} \leq 0
$$
for $|x|< \pi/2$. Indeed, the function $\log \cos x + \frac{x^2}{2}$ is even and vanishes at $x=0,$ but it is also concave since its second derivative $1-1/\cos^2x$ is negative on $0<|x| < \pi/2.$ Asymptotical sharpness of the upper bound in \eqref{binom-coeff} follows  by Stirling's formula.}\\

The following proposition establishes a connection between 
	the polar of the section $\widetilde{L}$ of a full-dimensional cone $L$ in $\cQ$
and 
	the section $\widetilde{L^\ast}$ of its dual in the $L^2$ metric.

\begin{proposition}
    \label{polar-dual}
	Let $L$ be a full-dimensional cone in $\cQ$ such that 
        $\big(\sum_{i=1}^n x_i^2\big)^2$
    is in the interior of $L$, and assume $\int_{S^{n-1}} f \dd\sigma>0$ for every nonzero $f\in L$.
    Then
        $$(\widetilde L)^\circ=-\widetilde{L^\ast}.$$
\end{proposition}

\begin{proof}
    We have that:
    \begin{align*}
            \widetilde{L^\ast}
            &=
            \left\{
                f\in \cM\colon
                f+(\sum_{i=1}^nx_i^2)^2\in L^\ast
            \right\}\\
            &=
            \left\{
                f\in \cM\colon
                \langle
                f+(\sum_{i=1}^nx_i^2)^2,
                h
                \rangle\geq 0 \quad
                \forall h\in L
            \right\}\\
            &=
            \left\{
                f\in \cM\colon
                \langle
                f+(\sum_{i=1}^nx_i^2)^2,
                h
                \rangle\geq 0 \quad
                \forall h\in L'
            \right\}\\
            &=
            \left\{
                f\in \cM\colon
                \langle
                f,
                h
                \rangle\geq -1 \quad
                \forall h\in L'
            \right\}\\
            &=
            \left\{
                f\in \cM\colon
                \langle
                -f,
                h
                \rangle\leq 1 \quad
                \forall h\in L'
            \right\}\\
            &=
            \left\{
                f\in \cM\colon
                \langle
                -f,
                g+(\sum_{i=1}^nx_i^2)^2
                \rangle\leq 1 \quad
                \forall g\in \widetilde L
            \right\}\\
            &=
            \left\{
                f\in \cM\colon
                \langle
                -f,
                g
                \rangle\leq 1 \quad
                \forall g\in \widetilde L
            \right\}\\
		&=
		-(\widetilde L)^\circ,
    \end{align*}
    where in the third equality we used 
    the homogeneity of the inner product,
    in the fourth the equality 
    $\big\langle (\sum_{i=1}^nx_i^2)^2,h\big\rangle=1$ 
    for every $h\in L'$
    and in the seventh the equality
    $\big\langle f,(\sum_{i=1}^nx_i^2)^2\big\rangle=0$ 
    for every $f\in \cM$.
    This concludes the proof of the proposition.
\end{proof}

Using Theorems \ref{BS}, \ref{revBS}
together with Lemma \ref{vrad-int-vs-d-metric} and Proposition \ref{polar-dual},
we obtain the following versions of the Blaschke-Santal\'o inequality and its reverse.

\begin{corollary}
    \label{BSr-in-d-metric}
    Let $L$ be a full-dimensional cone in $\cQ$ such that 
        $\big(\sum_{i=1}^n x_i^2\big)^2$
    is in the interior of $L$,
    and assume $\int_{S^{n-1}} f \dd\sigma>0$ for every nonzero $f\in L$.
    Then
    \begin{equation}
	\label{revBS-dif-metric}
	\frac{2}{(n+4)(n+6)}
        \leq 
        \vrad(\widetilde L)\vrad(\widetilde{L^{\ast}_d}).
    \end{equation}

    Moreover, if  $\big(\sum_{i=1}^n x_i^2\big)^2$
    is in addition the Santal\' o point of $L$, then
    \begin{equation}
	\label{BS-dif-metric}
        \vrad(\widetilde L)\vrad(\widetilde{L^{\ast}_d})
        \leq
        \Big(\frac{8}{(n+4)(n+6)}\Big)^{1-\frac{2n-1}{n^2+n-1}},
    \end{equation}
\end{corollary}

\begin{proof}
Using Theorem \ref{revBS} and Proposition \ref{polar-dual}
we have that 
\begin{equation}
    \label{130123-0908}
      {\frac14}\leq \vrad(\widetilde{L})\vrad(\widetilde{L^{\ast}}).
\end{equation}
Using 
	\eqref{130123-0908} 
and
	Lemma \ref{vrad-int-vs-d-metric}
implies \eqref{revBS-dif-metric}.
Using Theorem \ref{BS} instead of Theorem \ref{revBS} in the reasoning above we obtain the moreover part.
\end{proof}

\subsection{Rogers-Shepard inequality}
\label{RShepard}

Let $K$ be a bounded convex set in $\RR^n$ with a non-empty interior.
The \textbf{difference body} $\symm(K)$ of $K$ \cite{RS57} is defined by
    \begin{equation}
        \label{difference-body}
	\symm(K):=K-K.
    \end{equation}

The following inequality compares the volumes of $K$ and $\symm(K)$.

\begin{theorem}
    [{Rogers-Shepard inequality, \cite[Theorem 1]{RS57}}]
    \label{RS}
     Let $K$ be a bounded convex set in $\RR^n$ with a non-empty interior. 
	Then 
		$$\Vol(\symm(K))\leq \binom{2n}{n}\Vol(K)$$
	and hence
            $$\vrad(\symm(K))\leq 4\vrad(K).$$
\end{theorem}

\begin{color}{black}
\begin{remark}
        Let $K_1,K_2$ be any of the cones in \eqref{cones-of-forms-studied}.	
        In this remark we discuss one possible approach, 
        based on applying Theorem \ref{RS},
        to establish 
        asymptotic behavior of the ratio of volume radii of the compact bases $\widetilde K_1,\widetilde K_2$ as $n$ goes to infinity. 
        As we explain next, using this approach gives tight
        estimates only for cones contained in $\nn_\cQ$.

        Assume that $K_1\subseteq K_2$ and  
        $\widetilde{K}_2\subseteq c\cdot\symm(\widetilde K_1)$
        for some constant $c$.
        Then by Theorem \ref{RS},
	the ratio $\frac{\vrad{K_1}}{\vrad{K_2}}$, as $n$ goes to infinity, is 
	bounded below by $\frac{1}{4c}$.
        To prove that this lower bound is strictly positive
        as $n$ goes to infinity, one has to argue that
        $c$ can be chosen independently of $n.$

        To derive a dilation constant $c$ from the
        previous paragraph, it suffices to consider the extreme points of $\widetilde{K}_2$. Let $p$ be such an extreme point and let $c_p$ be the smallest constant 
	such that $p\in c_p\cdot\symm(\widetilde K_1)$.
        A good choice for $c$ is  then
	\[c=\sup\{c_p\colon p\text{ is an extreme point of }\widetilde{K}_2\}.\]

	Let $p$ be of the form 
			$c_{ij}x_i^2x_j^2-(\sum_{i=1}^{n}x_i^2)^2$ for $c_{ij}\in \RR$, 
        where $c_{ij}\in \RR$ is such that $p\in \widetilde{K}_2$. It turns out by a simple computation that $p$
        belongs to $4\symm(\CP_\cQ)$ (see \eqref{constant7} below). 
	Since all extreme points of $\widetilde{\nn}_\cQ$ are of this form, this implies that 
		$\widetilde{K}_2\subseteq 4\cdot\symm(\widetilde K_1)$
	for any pair of cones $K_1,K_2$ sandwiched between $\CP_\cQ$ and $\nn_\cQ$. 
 
	However, not all extreme points of $\widetilde\Pos_\cQ$ and $\widetilde\Sos_\cQ$ 
	are of the simple form from the previous paragraph. For such extreme points $p$ it is not clear whether 
    some dilation constant $c_p$ as above can be chosen
    independently of $n$.
    This  is the main limitation preventing us from establishing Theorem \ref{intro-vrad-of-our-sets}
    solely by applying Theorem \ref{RS}.
\end{remark}
\end{color}

\section{Volume radii estimates of our cones}
\label{volume-radii}

In this section we prove our main results (see Theorems \ref{intro-vrad-of-our-sets} and \ref{vrad-quartics}) on the estimates of volume radii of the cones under investigation:

\begin{theorem}
    \label{vrad-of-our-sets}
    Let 
    $$
    \cC:=\{\Pos_\cQ,\Sos_\cQ=\SPN_{\cQ},\nn_\cQ,\psd_\cQ,\DNN_\cQ,\Lf_\cQ,\CP_\cQ\}
    $$
	be the set of cones in the vector space of even quartics $\cQ$.
    We have that
    \begin{equation}
        \label{vrad-asymptotics-with-constants}
		(2^4\sqrt{2})^{-1}\cdot n^{-1}
        \leq \vrad(\widetilde {\CP_\cQ}) 
        \leq \vrad(\widetilde {\Pos_\cQ})
        \leq 2^3\cdot 3^2\cdot \sqrt{2}\cdot n^{-1}.
    \end{equation}
    In particular, for every $K\in \cC$ it holds that  
    \begin{equation*}
        \label{vrad-asymptotics}
            \vrad(\widetilde K)=\Theta(n^{-1})
            \qquad
            \text{and}
            \qquad
    \frac{1}{2^8\cdot 3^2}
    \leq
    \frac
    {\vrad(\widetilde{K})}{\vrad(\widetilde{\Pos_\cQ})}
    \leq 
    1,
    \end{equation*}
\end{theorem}

Before proving Theorem \ref{vrad-of-our-sets} we need three lemmas.
The first lemma compares the section 
	$\widetilde{\nn_\cQ}$ 
with the sections
	$\widetilde{\Lf_\cQ}$ and $\widetilde{\cp_\cQ}$.

\begin{lemma}
    \label{inclusion-of-sections}
    The following inclusions hold:
    \begin{enumerate}[\rm(1)]
    \item\label{inclusion-NN-Lf}
        $\widetilde{\Lf_\cQ}
            \subseteq \widetilde{\nn_\cQ}
            \subseteq 2\symm(\widetilde{\Lf_\cQ})$.
        \smallskip
    \item\label{inclusion-NN-CP}
        $\widetilde{\cp_\cQ}
            \subseteq \widetilde{\nn_\cQ}
            \subseteq 4\symm(\widetilde{\cp_\cQ})$.
        \medskip
    \end{enumerate}
\end{lemma}

\begin{proof}
    The first inclusions in 
        \ref{inclusion-NN-Lf} 
    and
        \ref{inclusion-NN-CP}
    are clear. 
    To prove the remaining inclusions of the lemma note that
    it suffices to prove that every extreme point of $\widetilde{\nn_\cQ}$
    is contained in the corresponding set.
    Note that the extreme points of $\widetilde{\nn_{\cQ}}$
    are of two types:
    \begin{align}
        \label{ext-nn-type1}
            &\displaystyle\frac{n(n+2)}{3}x_i^4-\big(\sum_{i=1}^{n}x_i^2)^2\quad
            \text{for some }i=1,\ldots,n,\\
        \label{ext-nn-type2}
            &\displaystyle n(n+2)x_i^2x_j^2-\big(\sum_{i=1}^{n}x_i^2)^2\quad
            \text{for some }i,j=1,\ldots,n,\ i\neq j.
    \end{align}
    The extreme points of the form \eqref{ext-nn-type1} clearly belong to
    both sections 
	$\widetilde{\Lf_\cQ}$, $\widetilde{\cp_\cQ}$,
    hence also to the dilations of their difference bodies.
    So it remains to study the extreme points of the form \eqref{ext-nn-type2}.
    For the case of $\widetilde{\Lf_\cQ}$ we have the following computation:
    \begin{align}
	  \label{constant-2}
	  \begin{split}
        &n(n+2)x_i^2x_j^2-\big(\sum_{i=1}^{n}x_i^2)^2=\\
        &=
        \frac{n(n+2)}{6} 
        \left(
        \pr_\cQ\big((x_i+x_j)^4-x_i^4-x_j^4\big)
        \right)
        -\big(\sum_{i=1}^{n}x_i^2)^2\\
        &=
        2\underbrace{\left(
        \frac{n(n+2)}{12}\pr_\cQ((x_i+x_j)^4)-\big(\sum_{i=1}^{n}x_i^2)^2
        \right)}_{p_1}
        -
       \frac{1}{2}
	   \underbrace{\left(
        \frac{n(n+2)}{3}x_i^4-\big(\sum_{i=1}^{n}x_i^2)^2
        \right)}_{p_2}\\
        &\hspace{1cm}-
        \frac{1}{2}\underbrace{\left(
        \frac{n(n+2)}{3}x_j^4-\big(\sum_{i=1}^{n}x_i^2)^2
        \right)}_{p_3}
	   \begin{color}{black}=p_1+\frac{1}{2}(p_1-p_2)+\frac{1}{2}(p_1-p_3)\end{color}\\
        &\in 
        \begin{color}{black}\widetilde{\Lf_\cQ}+
        \frac{1}{2}\symm(\widetilde{\Lf_\cQ})+
        \frac{1}{2}\symm(\widetilde{\Lf_\cQ})
        \subseteq
        2\symm(\widetilde{\Lf_\cQ}),\end{color}
	  \end{split}
    \end{align}
	where we used \eqref{integral-of-x4} and \eqref{integral-of-x2y2} in the containment of the last line.
	This concludes the proof of \ref{inclusion-NN-Lf}.
	Similarly for the case of $\widetilde{\cp_\cQ}$ the following computation holds:
	\begin{align}
	\label{constant7}
	\begin{split}
        &n(n+2)x_i^2x_j^2-\big(\sum_{i=1}^{n}x_i^2)^2=\\
        &=
        \frac{n(n+2)}{2} 
        \left(
        (x_i^2+x_j^2)^2-x_i^4-x_j^4\big)
        \right)
        -\big(\sum_{i=1}^{n}x_i^2)^2\\
        &=
        4\underbrace{\left(
        \frac{n(n+2)}{8}(x_i^2+x_j^2)^2)-
        \big(\sum_{i=1}^{n}x_i^2)^2
        \right)}_{p_4}
        -
        \frac{3}{2}\left(
        \frac{n(n+2)}{3}x_i^4-\big(\sum_{i=1}^{n}x_i^2)^2
        \right)\\
        &\hspace{1cm}-
        \frac{3}{2}\left(
        \frac{n(n+2)}{3}x_j^4-\big(\sum_{i=1}^{n}x_i^2)^2
        \right)=\begin{color}{black}p_4+\frac{3}{2}(p_4-p_2)+\frac{3}{2}(p_4-p_3)\end{color}\\
        &\in 
        \begin{color}{black}\widetilde{\cp_\cQ}+
        \frac{3}{2}\symm(\widetilde{\cp_\cQ})+
        \frac{3}{2}\symm(\widetilde{\cp_\cQ})
        \subseteq 
        4\symm(\widetilde{\cp_\cQ}),\end{color}
	\end{split}
    \end{align}
	\begin{color}{black}where $p_2,p_3$ are as in \eqref{constant-2}\end{color} and we used \eqref{integral-of-x4} and \eqref{integral-of-x2y2} in the containment of the last line.
	This concludes the proof of \ref{inclusion-NN-CP} and the lemma.
\end{proof}

\begin{remark}
	By the same reasoning as in the proof of Lemma \ref{inclusion-of-sections}
	one can show that 
		$$\widetilde{\nn_\cQ}\subseteq 2\symm(\widetilde{\psd_\cQ}).$$
	Indeed, for the extreme points of $\widetilde{\nn_\cQ}$ of the form \eqref{ext-nn-type2} 
	this inclusion follows by the following computation:
\begin{align*}
        &n(n+2)x_i^2x_j^2-\big(\sum_{i=1}^{n}x_i^2)^2=\\
        &=
        \frac{n(n+2)}{4} 
        \left(
        (x_i^2+x_j^2)^2-(x_i^2-x_j^2)^2
        \right)
        -\big(\sum_{i=1}^{n}x_i^2)^2\\
        &=
        2\underbrace{\left(
        \frac{n(n+2)}{8}(x_i^2+x_j^2)^2-
        \big(\sum_{i=1}^{n}x_i^2)^2
        \right)}_{p_1}
        -
        \underbrace{\left(
        \frac{n(n+2)}{4}(x_i^2-x_j^2)^2-
        \big(\sum_{i=1}^{n}x_i^2)^2
        \right)}_{p_2}\\
        &=p_1+(p_1-p_2) 
        \in \widetilde{\psd_\cQ}+
        \symm(\widetilde{\psd_\cQ})
        \subseteq 
        2\symm(\widetilde{\psd_\cQ}).
\end{align*}
\end{remark}

The second lemma needed in the proof of Theorem \ref{vrad-of-our-sets} 
establishes two dualities in the differential metric between the sections of the cones from 
Theorem \ref{vrad-of-our-sets}.

\begin{lemma}
   \label{dualities-diff-metric}
    We have the following dualities in the differential metric:
    \begin{enumerate}[\rm(1)]
    \item
        \label{NN-self-dual-in-d}
        $\widetilde{(\nn_\cQ)_d^\ast}=\widetilde{\nn_\cQ}.$
    \medskip
    \item
        \label{dual-of-Pos-in-d}
        $\widetilde{(\Lf_\cQ)_d^\ast}=\widetilde{\Pos_\cQ}.$
    \end{enumerate}
\end{lemma}

\begin{proof}
    First we prove \ref{NN-self-dual-in-d}.
    It is equivalent to establish 
        $(\nn_\cQ)_d^\ast=\nn_\cQ$.
    We have:
    \begin{align*}
    (\nn_\cQ)_d^\ast
    &=
    \left\{
        \sum_{i,j=1}^n a_{ij}x_i^2x_j^2\in \cQ\colon
            \big\langle 
                \sum_{i,j=1}^n a_{ij}x_i^2x_j^2,g
            \big\rangle_d\geq 0\quad \forall g\in \nn_\cQ
    \right\}\\
    &=
    \left\{
        \sum_{i,j=1}^n a_{ij}x_i^2x_j^2\in \cQ\colon
            \big\langle 
                \sum_{i,j=1}^n a_{ij}x_i^2x_j^2,x_k^2x_\ell^2
            \big\rangle_d\geq 0\quad 
            \forall k,\ell=1,\ldots,n
    \right\}\\
    &=\left\{
        \sum_{i,j=1}^n a_{ij}x_i^2x_j^2\in \cQ\colon
            a_{k\ell}\geq 0\quad 
            \forall k,\ell=1,\ldots,n
    \right\}\\
    &=\nn_\cQ,
    \end{align*}
    where in the second equality we used that the extreme points 
    of $\nn_\cQ$ are of the form $c^2x_k^2x_\ell^2$ for some $c\in \RR$ and $k,\ell=1,\ldots,n$,
    while in the third equality we used that
    $\big\langle 
                \sum_{i,j=1}^n a_{ij}x_i^2x_j^2,x_k^2x_\ell^2
     \big\rangle_d
     =24a_{k\ell} 
     $
     if $k=\ell$
     and
     $4(a_{k\ell}+a_{\ell k})=8a_{k\ell}=8a_{\ell k}$ otherwise.

    It remains to prove \ref{dual-of-Pos-in-d}.
    This easily follows by observing that for $f\in \cQ$ we have
    $$
        \Big\langle 
            f,
            \pr_\cQ\Big(\big(\sum_{i=1}^n v_ix_i\big)^4\Big)
        \Big\rangle_d
        =
        24f(v),
    $$
    for any $v=(v_1,\ldots,v_n)\in \RR^n$.
\end{proof}

The third lemma needed in the proof of Theorem \ref{vrad-of-our-sets} 
identifies the Santal\'o point of $\widetilde{\Lf_\cQ}$.

\begin{lemma}
	\label{Santalo-of-Lf}
		The Santal\'o point of $\widetilde{\Lf_\cQ}$ is the origin.
\end{lemma}

\begin{proof}
	Every element $O$ of the orthogonal group $O(n)$ defines a linear map 
	\begin{equation}
		\label{def-of-O-action}
		L_O:\cQ\to\cQ,\qquad (L_Of)(\x):=\pr_\cQ\big(f(O\x)\big),
	\end{equation}
	where $\pr_\cQ$ is defined as in \eqref{proj-Q}.

	We will prove that $\widetilde{\Lf_\cQ}$ is invariant under every map $L_O$, $O\in O(n)$ and 
	the origin is the only fixed point of $\widetilde{\Lf_\cQ}$.
	Since the Santal\'o point of a convex body is unique, the statement of the lemma will follow from these two facts.
	
	First we prove 
		$$r(\x):=\big(\sum_{i=1}^n x_i^2\big)^2$$
	is a fixed point of every map $L_O$ defined by \eqref{def-of-O-action}.
	Since $r(O\x)\equiv 1$ on $S^{n-1}$ for every $O\in O(n)$ and $r(\x)$ is the only form from $\RR[\x]_4$
	such that $r(\x)\equiv 1$ on $S^{n-1}$, it follows that 
		$r(O\x)=r(\x)$ for every $O\in O(n)$.
	Hence, 
	\begin{equation}
		\label{r-fixed}
			(L_Or)(\x)=\pr_\cQ\big(r(O\x)\big)=\pr_\cQ\big(r(\x)\big)=r(\x)
	\end{equation}
	and $r(\x)$ is indeed a fixed point of every map $L_O$ defined by \eqref{def-of-O-action}.

	Next we prove that $\widetilde{\Lf_\cQ}$ is invariant for every map $L_O$ defined by \eqref{def-of-O-action}.
	Let us choose an arbitrary $g\in\widetilde{\Lf_\cQ}$. Then $g$ is of the form
	\begin{equation*}
		\label{form-of-g}
			g(\x)=\pr_\cQ\big(\sum_{i} f_i^4\big)-r(\x),
	\end{equation*}
	where $f_i\in \RR[\x]_1$.
	For $O\in O(n)$ we have that 
	\begin{align}
		\label{action-on-g}
		\begin{split}
		(L_Og)(\x)
		&=\pr_\cQ\Big(\pr_\cQ\big(\sum_{i} (f_i(O\x))^4\big)-r(O\x)\Big)\\
		&=\pr_\cQ\Big(\sum_{i} (f_i(O\x))^4\Big)-r(\x)\in \widetilde{\Lf_\cQ},
		\end{split}
	\end{align}
	where we used \eqref{r-fixed} in the second equality, 
	while for the containment we used the fact that 
	$f_i(O\x)\in \RR[\x]_1$ for every $O\in O(n)$.
	Since $g$ was arbitrary, \eqref{action-on-g} proves that 
	$\widetilde{\Lf_\cQ}$ is indeed invariant for every map $L_O$ defined by \eqref{def-of-O-action}.

	It remains to prove that the origin is the only fixed point of $\widetilde{\Lf_\cQ}$ for every map $L_O$, $O\in O(n)$.
	It suffices to prove that fixed points of $\cQ$ are of the form $cr(\x)$ for $c\in \RR$, since the only $c$ such that
	$cr(\x)-r(\x)\in \cM$ is equal to 1.
	So  let 
	$$
		f(\x)=\sum_{1\leq i \leq j\leq n}a_{ij}x_i^2x_j^2\in \cQ
	$$
	be a fixed point (here we used that the coefficients $a_{ij}$ and $a_{ji}$ in \eqref{101222-1925} are the same for every form in $\cQ$).
	Let $\pi_{k\ell}=(k\; \ell)$, $1\leq k<\ell\leq n$, be a transposition. We have
	\begin{align}
		\label{act-by-transpose}
		\begin{split}
		(L_{\pi_{k\ell}}f)(\x)
			&=
			a_{kk}x_\ell^4+a_{\ell\ell}x_k^4+
			\sum_{i\notin\{k,\ell\}} a_{ii}x_i^4+
			a_{k\ell}x_k^2x_\ell^2+
			\sum_{i<k} a_{ik} x_{i}^2x_{\ell}^2+
			\sum_{\substack{j> k,\\ j\neq \ell}} a_{kj} x_{\ell}^2x_{j}^2\\
			&\hspace{1cm}+
			\sum_{\substack{i< \ell,\\ i\neq k}} a_{i \ell} x_{i}^2x_{k}^2+
			\sum_{j>\ell} a_{\ell j} x_{k}^2x_{j}^2+
			\sum_{i,j\notin\{k,\ell\}}a_{ij}x_i^2x_j^2.
		\end{split}
	\end{align}
	Since $L_{\pi_{k\ell}}f=f$, 
	it follows in particular from \eqref{act-by-transpose} 
	by considering the coefficient at $x_k^4$
		that $a_{\ell\ell}=a_{kk}$,
	while the comparison of the coefficients at $x_{k}^2 x^2_{m}$ for $m\notin \{k,\ell\}$ implies that
		$a_{\ell m}=a_{k m}$ if $k<m$ and $a_{m\ell}=a_{mk}$ if $k>m$.
	Note that $k,\ell$ and $m\notin \{k,\ell\}$ were arbitrary. So we conclude that
	there are constants $c,d\in\RR$ such that
	$$
		a_{ii}=c\quad \text{for all }i=1,\ldots,n
		\qquad\text{and}\qquad 
		a_{ij}=d\quad \text{for all }1\leq i<j\leq n.
	$$
	Hence, 
	$$
		f(\x)
		=c\Big(\sum_{i=1}^n x_i^4\Big)+d\Big(\sum_{1\leq i<j\leq n}x_i^2x_j^2\Big).
	$$
	Let $O_{k\ell}$, $1\leq k< \ell\leq n$, be an orthogonal transformation defined on the 
	standard basis vectors $e_i$, $i=1,\ldots,n$, having the only nonzero entry in the $i$-th coordinate which is 1,
	by 
		$e_k\mapsto \frac{1}{\sqrt{2}}(e_{k}-e_\ell)$, 
		$e_\ell\mapsto \frac{1}{\sqrt{2}}(e_{k}+e_\ell)$
	and
		$e_i\mapsto e_i$ if $i\notin \{k,\ell\}$.
	We have that
	\begin{align}
		\label{act-by-orthogonal}
		\begin{split}
		(L_{O_{k\ell}}f)(\x)	
		&=
		\pr_\cQ\Bigg(
		c
		\Big( 
			\frac{1}{4}(x_k-x_\ell)^4+\frac{1}{4}(x_k+x_\ell)^4+\sum_{i\notin \{k,\ell\}} x_i^4
		\Big)\\
		&\hspace{1cm}
		+d
		\Big(\frac{1}{4} (x_k^2-x_\ell^2)^2+
			\sum_{i<k} \frac{1}{2} x_{i}^2(x_k-x_{\ell})^2+
			\sum_{\substack{j> k,\\ j\neq \ell}} \frac{1}{2} (x_k-x_{\ell})^2x_{j}^2\\
			&\hspace{1cm}+
			\sum_{\substack{i< \ell,\\ i\neq k}} \frac{1}{2} x_{i}^2(x_{k}+x_\ell)^2+
			\sum_{j>\ell} \frac{1}{2} (x_{k}+x_\ell)^2x_{j}^2+
			\sum_{i,j\notin\{k,\ell\}}x_i^2x_j^2
		\Big)
		\Bigg).
	\end{split}
	\end{align}
	Since $L_{O_{k\ell}}f=f$, 
	it follows in particular from \eqref{act-by-orthogonal}
	by considering the coefficient at $x_k^4$
		that $c=\frac{1}{2}c+\frac{1}{4}d$, or equivalently, $d=2c$.
	Hence,
		$$f(\x)=c\Big(\sum_{i=1}^n x_i^4+2\sum_{1\leq i<j\leq n}x_i^2x_j^2\Big)=cr(\x).$$
	This concludes the proof of the lemma.
\end{proof}

Finally we can prove Theorem \ref{vrad-of-our-sets}.


\begin{proof}[Proof of Theorem \ref{vrad-of-our-sets}]
Using \eqref{revBS-dif-metric} of Corollary \ref{BSr-in-d-metric} for $L=\nn_\cQ$ and
\ref{NN-self-dual-in-d} of Lemma \ref{dualities-diff-metric} 
implies that
\begin{equation}
    \label{120123-2028}
      {\frac{2}{(n+4)(n+6)}}
      \leq 
      \big(\vrad(\widetilde{\nn_\cQ})\big)^2.
\end{equation}
\smallskip

\noindent\textbf{Claim 1.} For $n\geq 5$ we have that 
\begin{equation}
	\label{230508-1825}
		\displaystyle\frac{1}{2n^2}\leq \frac{2}{(n+4)(n+6)}.
\end{equation}
\smallskip

\noindent\textit{Proof of Claim 1.} 
Multiplying \eqref{230508-1825} by $2n^2(n+4)(n+6)$ and rearranging terms, it follows that \eqref{230508-1825}
is equivalent to
\begin{equation}
	\label{230508-1826}
		3n^2-10n-24\geq 0.
\end{equation}
Since the local minimum $x_0$ of the quadratic function $f(x)=3x^2-10x-24$ is equal to $x_0=\frac{5}{3}$ and $f(5)=1$,
this in particular implies that \eqref{230508-1826} holds true. \hfill$\square$

\bigskip
 
Now \eqref{120123-2028} and Claim 1 imply that
\begin{equation}
    \label{lower-bound-NN}
      (\sqrt{2}n)^{-1}
      \leq 
     \vrad(\widetilde{\nn_\cQ})
\end{equation}
Using Theorem \ref{RS} for 
    $K=\widetilde{\Lf_\cQ}$ (resp.\ $K=\widetilde{\CP_\cQ}$)
together with 
    \ref{inclusion-NN-Lf} (resp.\ \ref{inclusion-NN-CP})
    of Lemma \ref{inclusion-of-sections}
gives
\begin{equation}
    \label{vrad-bounds-Lf}
    {
    (8\sqrt{2}n)^{-1}
	}
    \leq 
    \frac{1}{8}\vrad(\widetilde{\nn_{\cQ}})
    \leq
    \vrad(\widetilde{\Lf_{\cQ}}),
\end{equation}
\begin{equation}
    \label{vrad-bounds-cp}
    {
	(16\sqrt{2}n)^{-1}
	}
	\leq
    \frac{1}{16}\vrad(\widetilde{\nn_{\cQ}})
    \leq
    \vrad(\widetilde{\cp_{\cQ}}).
\end{equation}
Since by Lemma \ref{Santalo-of-Lf} the origin is the Santal\'o point of $\widetilde{\Lf_\cQ}$,
using  \eqref{BS-dif-metric} of Corollary \ref{BSr-in-d-metric} for $L=\Lf_\cQ$ and
\ref{dual-of-Pos-in-d} of Lemma \ref{dualities-diff-metric} implies that
    \begin{equation}
    \label{130123-0005}
        \vrad(\widetilde{\Pos_\cQ})
        \leq 
        \Big(\frac{8}{(n+4)(n+6)}\Big)^{1-\frac{2n-1}{n^2+n-1}}
        (\vrad(\widetilde{\Lf_\cQ}))^{-1}.
    \end{equation}
\smallskip

\noindent\textbf{Claim 2.} For $n\geq 3$ we have that 
\begin{equation}
	\label{230508-2055}
		\Big(\frac{8}{(n+4)(n+6)}\Big)^{1-\frac{2n-1}{n^2+n-1}}
		\leq
		\frac{9}{n^2}.
\end{equation}
\smallskip

Before we prove Claim 2 we establish two preliminary results.
\bigskip

\noindent\textbf{Claim 2.1.} For $n\in \NN$ we have that 
\begin{equation}
	\label{230508-2110}
		\frac{2n-1}{n^2+n-1}\leq \frac{2}{n}.
\end{equation}
\smallskip

\noindent\textit{Proof of Claim 2.1.} 
Multiplying \eqref{230508-2110} by $n(n^2+n-1)$ and rearranging terms, it follows that \eqref{230508-2110}
is equivalent to $3n-2\geq 0$, which clearly implies Claim 2.1. \hfill$\square$

\bigskip

\noindent\textbf{Claim 2.2.} For $n\in \NN$ we have that 
\begin{equation}
	\label{230508-2115}
		\frac{8n^2}{(n+4)^2}
        \Big(\frac{(n+4)^2}{8}\Big)^{\frac{2}{n}}
        < 9.
\end{equation}
\smallskip

\noindent\textit{Proof of Claim 2.2.} 
Let $g(x)=\frac{8x^2}{(x+4)^2}
        \Big(\frac{(x+4)^2}{8}\Big)^{\frac{2}{x}}$.
We have that 
$$g'(x)
=
\frac{2^{4-\frac{6}{x}} \big((x+4)^2\big)^{\frac{2}{x}} 
\Big(x (6+\log 8)-(x+4) \log \left((x+4)^2\right)+4 \log 8\Big)}{(x+4)^3}.
$$
Note that for $x>0$, 
$g'(x)\geq 0$
iff
$h(x)\geq 0$,
where 
$$
h(x)=x (6+\log 8)-(x+4) \log \left((x+4)^2\right)+4 \log 8.
$$
Since 
$$h'(x)=4+\log 8-\log((x+4)^2),$$
it follows that $h(x)$ is increasing on 
$(0,x_m]$, where $x_m:=2\sqrt{2}e^{2}-4\approx 16.9$ and 
decreasing for $x>x_m$.
Since 
    $$h(1)=6+5\log 8-5\log 25\approx 0.3>0,$$ 
it follows that $h(x)$ has only one zero $x_0$ 
on the interval $[1,\infty)$, where $x_0>x_m$.
From $h(38)\approx 1.37>0$ and $h(39)\approx -0.047$, it follows
that $x_0\in (38,39)$.
Finally, $g(38)\approx 8.6995<9$ and $g(39)=8.6997<9$, 
proves Claim 2.2.\hfill$\square$
\bigskip

Now we are ready to prove Claim 2.
\bigskip

\noindent\textit{Proof of Claim 2.}
We have that
\begin{align*}
n^2\Big(\frac{8}{(n+4)(n+6)}\Big)^{1-\frac{2n-1}{n^2+n-1}} 
&\leq 
n^2\Big(\frac{8}{(n+4)^2}\Big)^{1-\frac{2n-1}{n^2+n-1}}
\leq
\frac{8n^2}{(n+4)^2}
\Big(\frac{(n+4)^2}{8}\Big)^{\frac{2}{n}}
<9,
\end{align*}
where we used that $n+6>n+4$ and
$1-\frac{2n-1}{n^2+n-1}>0$ for $n\geq 3$ (Claim 2.1) in the first inequality,
$\frac{2n-1}{n^2+n-1}\leq \frac{2}{n}$ in the second
and Claim 2.2 in the third.
\hfill$\square$

\smallskip

Using \eqref{vrad-bounds-Lf} and Claim 2 in \eqref{130123-0005}, it follows that
    \begin{equation}
    \label{vrad-pos-upper-bound}
        \vrad(\widetilde{\Pos_\cQ})
        \leq 
	2^3\cdot3^2\cdot\sqrt{2} \cdot n^{-1}.
    \end{equation}
Now \eqref{vrad-bounds-cp} and \eqref{vrad-pos-upper-bound} imply \eqref{vrad-asymptotics-with-constants} holds,
which proves the theorem.
\end{proof}

\end{document}